\documentclass[11pt,a4paper]{scrartcl}
\usepackage[utf8]{inputenc}
\usepackage[english]{babel}
\usepackage{helvet}
\usepackage{amsmath,empheq}
\usepackage{amsthm}
\usepackage{amsfonts}
\usepackage{amssymb}
\usepackage{mathtools}
\usepackage{mathrsfs}
\usepackage{dsfont}
\usepackage{hyperref}
\usepackage[noabbrev,capitalize]{cleveref}
\usepackage{tabularx}
\usepackage{xcolor}
\usepackage[paper=a4paper,left=25mm,right=25mm,top=25mm,bottom=25mm]{geometry}

\author{Ana Djurdjevac, Helena Kremp, Nicolas Perkowski}

\newtheorem{theorem}{Theorem}[section]
\newtheorem{definition}[theorem]{Definition}     
\newtheorem{proposition}[theorem]{Proposition}
\newtheorem{lemma}[theorem]{Lemma}	

\newtheorem{corollary}[theorem]{Corollary}

\newtheorem{remark}[theorem]{Remark}

\numberwithin{equation}{section}

\newcommand{\R}{\mathbb{R}} 
\newcommand{\N}{\mathbb{N}} 
\newcommand{\E}{\mathds{E}} 
\newcommand{\p}{\mathds{P}} 
\newcommand{\F}{\mathcal{F}} 

\renewcommand{\mathcal}[1]{\mathscr{#1}}
\renewcommand{\epsilon}{\varepsilon}

\newcommand{\squeeze}[2][0]{\mbox{$\medmuskip=#1mu\displaystyle#2$}}

\DeclarePairedDelimiter{\abs}{\lvert}{\rvert}
\DeclarePairedDelimiter{\norm}{\lVert}{\rVert}

\DeclarePairedDelimiter{\angles}{\langle}{\rangle}

\DeclarePairedDelimiter{\paren}{(}{)}

\begin{document}

\title{Rough homogenization for Langevin dynamics on fluctuating Helfrich surfaces}

\author{Ana Djurdjevac \footnote{Freie Universität Berlin, Arnimallee 9, 14195 Berlin; adjurdjevac@zedat.fu-berlin.de } 
\and Helena Kremp \footnote{Freie Universität Berlin, Arnimallee 7, 14195 Berlin; helena.kremp@fu-berlin.de} 
\and Nicolas Perkowski \footnote{Freie Universität Berlin, Arnimallee 7, 14195 Berlin; perkowski@math.fu-berlin.de}}

\maketitle
\begin{abstract}
\noindent In this paper, we study different scaling rough path limit regimes in space and time for the Langevin dynamics on a quasi-planar fluctuating Helfrich surfaces. The convergence results of the processes were already proven in \cite{dep}. We extend this work by proving the convergence of the Itô and Stratonovich rough path lift.  
For the rough path limit, there appears, typically, an area correction term to the Itô  iterated integrals, and in certain regimes to the Stratonovich iterated integrals.
This yields additional information on the homogenization limit and enables to conclude on homogenization results for diffusions driven by the Brownian motion on the membrane using the continuity of the Itô-Lyons map in rough paths topology. 

\end{abstract}

\begin{section}{Introduction}

The lateral diffusion of particles is crucial
for cellular processes, including signal transmission, cellular organization and the
transport of matter (cf. \cite{mika2011macromolecule, mourao2014connecting, luby2013physical, almeida1995lateral}).  Motivated by these applications, i.e.  diffusion on cell membranes, we  consider the diffusion on a curved domain - a hypersurface $\mathcal{S}$, cf \cite{seifert1997configurations}. The Brownian motion on the surface, whose generator in local coordinates is the Laplace-Beltrami operator, is a simple example of a diffusing particle on a biological surface, which is also known from physics as the overdamped Langevin dynamics on a Helfrich membrane (cf. \cite{nb}). 

We will restrict our considerations to the classical situation of the so-called ``essentially 
flat surfaces'' $\mathcal{S}$. The standard way of representing the essentially flat surface is the Monge-gauge parametrization, where we specify the height $H$ of the hypersurface as a function of the coordinates from the flat base, namely over $[0,L]^2$ or, since we consider the periodic setting, over  $\mathbb{T}^{2}$ (such that we take $L=1$).

Moreover, these membranes are fluctuating both in time and space due to the spatial micro-structure and thermal fluctuations or active proteins.  
The analysis of the  macroscopic behavior of a laterally diffusive process on surfaces possessing microscopic space  and time scales was derived in \cite{duncan, dep}. Based on classical methods from homogenization theory, they prove that under the assumption of scale separation between the characteristic length and time scales of the
membrane fluctuations and the characteristic scale of the diffusing particle, the lateral
diffusion process can be well approximated by a Brownian motion on the plane with
constant diffusion tensor $D$. In particular, they show that $D$  depends in a highly nonlinear way on the detailed properties of the surface.

Since this work is motivated  by the Helfrich elasticity membrane model,  we briefly describe it here in order to better understand  understand the form of the  considered system of SDEs.

The classical description of fluid membranes $\mathcal{S}$ on the level of continuum elasticity is based on the   Canham \cite{canham1970minimum} -  Helfrich \cite{helfrich1973elastic} free energy:
\[
\mathcal{E}[\mathcal{S}] =  \frac{1}{2} \int_{[0,L]^2} \kappa K^2(x) dx 
\sqrt{|G|(x)}dx, 
\]
where $G$ is the metric tensor of $\mathcal{S}$ in local coordinates, $K$ is the mean and the constant $\kappa$ is the (bare) bending modulus. Note that the term with Gaussian curvature is omitted since we consider fluctuations of the membrane which do not change its topology. For more details about description of fluid lipid membranes see \cite{deserno2015fluid}.

Utilizing the standard approach (see \cite{dep, doi1988theory}) one can derive the dynamics of the surface fluctuations that are described by the stochastic partial differential equation (SPDE) of the type:
\begin{equation}\label{membrane}
\frac{dH(t)}{dt} = -RAH(t) + \xi(t)
\end{equation}
where $AH:=-\kappa\Delta^{2}H+\sigma\Delta H$ ($\sigma$ being the surface tension) is the restoring force for the free energy associated to $\mathcal{E}[\mathcal{S}]$. Moreover, $R$ is the operator that characterizes the effect of nonlocal interactions of the membrane
through the medium (for more details see \cite[section 4]{dep} or \cite[section 2.2]{duncan}; $Rf:=\Lambda\ast f$ for $\Lambda(x):=(8\pi\lambda\abs{x})^{-1}$, $f\in L^{2}_{\text{per}}([0,L]^{2})$, where $\lambda$ is the viscosity of the surrounding medium). The last term $\xi$ is a Gaussian field, that is white in time and with spatial fluctuations having mean zero and covariance operator $2(k_{B}T)R$, $k_{B}$ being the Boltzmann constant. Next, we consider the ansazt for the height function as the truncated  Galerkin-projection on the Fourier spanned space given by
\begin{align}\label{Hansatz}
H(x,t)=h(x,\eta_{t})=\sum_{\abs{k}\leqslant \tilde{K}}\eta_{t}^{k}e_{k}(x),\quad (x,t)\in\mathbb{T}^{2}\times\R^{+},
\end{align}
for the Fourier basis $(e_{k})_{k\in\mathbb{Z}^{2}}$ on the torus $\mathbb{T}^{2}$ and fixed cut-off $\tilde{K}\in\N$. Substituting the \cref{Hansatz}   into (\ref{membrane}) we see that the SPDE diagonalizes and that  coefficients $\eta=(\eta^{k})_{\abs{k}\leqslant K}$ with $K:= \#\{k\in\mathbb{Z}^{2}\mid 0<\abs{k}\leqslant\tilde{K}\}$ are independent Ornstein Uhlenbeck processes given by
\begin{align*}
d\eta^{k}_{t}=-\frac{\kappa\abs{2\pi k}^{3}+\sigma\abs{2\pi k}}{4\lambda}\eta^{k}_{t}dt+\sqrt{\frac{k_{B}T}{2\lambda\abs{2\pi k}}}dW^{k}_{t},
\end{align*} 
where $(W^{k})_{k}$ are independent complex-valued standard Brownian motions with constraint $\overline{W^{k}}=W^{-k}$, i.e. $W^{k}=\frac{1}{\sqrt{2}}(Re(W^{k})+i Im(W^{k}))$ with $Re(W^{k}),Im(W^{k})$ independent $\R^{K}$-valued Brownian motions. Moreover, using \cite[Lemma 6.25]{Stuart10}, we notice that without the ultra-violett cutoff $\tilde{K}$, the surface $H$ is almost surely Hölder-continuous with exponent $\alpha<1$ (in $d=2$), but not for $\alpha=1$. Due to this irregularity,  we can not define a diffusion on $H$ using the classical methods. This is the reason why most of the works that deal with the diffusion on $H$, assume the ultra-violet cutoff. 

Next, we consider a  particle on the surface $\mathcal{S}$ with the molecular diffusion $D_0$.
After the nondimensionalisation, we can  derive the equation that describes its dynamics (more comments will be given in the next section and the detailed derivation can be found in \cite[section 2.1]{duncan}. All together,  \cite[section 2.2]{duncan} one derives the system of coupled SDEs that describes the motion of the particle and fluctuating membrane:
\begin{align}
dX^\epsilon_t &= F(X^\epsilon_t , \eta^\epsilon_t)dt + \sqrt{\Sigma(X^\epsilon_t , \eta^\epsilon_t)} dB(t) \\
d\eta^\epsilon_t &= -\frac{1}{\epsilon} \Gamma \eta^\epsilon_t dt + \sqrt{\frac{2\Gamma \Pi}{\epsilon}} dW(t),
\end{align}
where $\epsilon :=4 D_0 \lambda/(k_{B}T) $  
 and where  $B$ is a standard Brownian motion independent of $W$. The coefficients $\Gamma,\Pi, F$ and $\Sigma$ will be specified later. 
Based on this particular setting, in \cite{dep} the authors generalize this model \eqref{eq:rsystem} to different $(\alpha, \beta)$ space-time scaling regimes and this will be the model that will be considered in this work in the rough path limit sense.

Motivated by the recent link between stochastic homogenization and rough paths from the works \cite{kelly2016smooth, fgl, rhfriz, radditive},
we prove a rough homogenization result for a Brownian particle on a fluctuating Gaussian hypersurface with covariance given by (the ultra-violet cutoff of) the Helfrich energy. Specifically, we extend previously mentioned results with proving the convergence towards a particular lift of the homogenization limit in rough path topology  for different scaling regimes $(\alpha,\beta)$. Interestingly, in some regimes, the rough path lift of $(X^{\epsilon})$ converges to a non-trivial lift of the limiting Brownian motion $X$, in the sense that, additionally, an area correction  to the iterated integrals appears. This phenomenon was already observed by \cite{ll} and \cite{fgl} in different situations. Considering the scaling limits in the rough path topology yields more information on the homogenization limit and  enables us to conclude homogenization results for diffusions driven by the Brownian motion on the membrane using the continuity of the Itô-Lyons map in rough path topology.

As in the works by \cite{duncan,dep,radditive}, we utilize martingale methods for additive functionals of Markov processes (cf. also \cite{klo}). That is, we identify a stationary, ergodic Markov process and exploit the solution of the associated Poisson equation for the generator of that Markov process to rewrite the (in general unbounded) drift term of $X^{\epsilon}$. We consider both Itô and Stratonovich rough path lifts. As already observed in \cite{radditive}, for the Stratonovich lift we expect an area correction to appear if and only if the underlying Markov process is non-reversible. Indeed, in the regime $(\alpha,\beta)=(1,1)$, we have that the underlying Markov process $Y^{\epsilon}:=\epsilon^{-1}X^{\epsilon}$ is reversible (under its invariant measure, for each fixed, stationary realization of $\eta$), and also the Stratonovich limit is the usual Stratonovich lift of the Brownian motion $X$. Contrary to that, in the $(\alpha,\beta)=(1,2)$ regime, an area correction for the Stratonovich lift appears (which we expect is truly non-vanishing). In the regime $(\alpha,\beta)=(0,1)$ the limit is obtained by averaging over the invariant measure of the Ornstein Uhlenbeck process $\eta$; and also the rough limit is given by the canonical lift of the Brownian motion, which follows as in this case the uniform controlled variation (UCV) condition is satisfied  for $(X^\epsilon)$. 

In the regime $(\alpha,\beta)=(1,-\infty)$ and $\eta_{0}$ being deterministic (i.e. $H$ being a non-random periodic surface), we have that $X$ is a diffusion with periodic coefficients and the equality in law $X^{\epsilon}\stackrel{d}{=}(\epsilon X_{\epsilon^{-2}t})_{t}$ holds. Note that this case corresponds to the quenched regime, i.e.  we consider the surface to be time-independent (or fix a stationary realization $\eta_{0}$).  Due to the functional central limit theorem for the Itô rough path lift of $X^{\epsilon}$ proven in \cite[section 4.3]{radditive} the rough path limit is given by a nontrivial lift of the limiting Brownian motion $X$.

Rough homogenization results are useful in combination with the continuity of the Itô-Lyons map to obtain rough homogenization results for particles whose velocity is governed by the velocity of the particle $X$.
For example,  processes like fusion are initiated by the velocities of  particles itself. However, note that in this case the underdamped Langevin dynamics would better describe  the system,  and in our work we consider the overdamped Langevin dynamics. Hence, our results could be seen as a starting point for considering these more complicated systems. 

The paper is structured as follows. \cref{sect:prelim} sets the model and the assumptions on the surface $H$. We also collect knowledge about the generator of the Markov process $(X,\eta)$ and the growth conditions on the coefficients, as well as recall the definition of the space of $\alpha$-Hölder rough path. \cref{sect:averaging}, \cref{sect:hom1} and \cref{sect:hom2} treat the scaling regimes $(\alpha,\beta)=(0,1)$, $(\alpha,\beta)=(1,2)$ and $(\alpha,\beta)=(1,1)$, respectively. For every case we prove tightness in $\gamma$-Hölder rough path topology for $\gamma \in (\frac{1}{3},\frac{1}{2})$. As a consequence, we derive the rough homogenization limit.  
\end{section}

\begin{section}{Preliminaries}\label{sect:prelim}

In this section we will fix the definition of the time dependent random hypersurface $H(x,t)$ and the diffusion $X$ on $H$.  For more details, we refer to \cite{duncan, dep}.

We will consider fluctuating hypersurfaces that can be represented as a graph of a sufficiently smooth random field  $H: [0,L]^d \times [0,\infty) \to \mathbb{R}$, i.e. it is represented via the so-called Monge gauge parametrization. More precisely, we assume that for each
$t>0$, $H(x; t)$ is smooth in $x$ and $H(x; t)$ is periodic in $x$ with period $L_H$ for each $t >0$. Without loss of generality, we assume that $L_{H}=1$.
Furthermore, we assume the existence of a characteristic timescale $T_H=T$, which models the observation time of the system. The hypersurface $\mathcal{S}(t)$ is parametrized over $[0,1]^d$  by $J: [0,1]^d \times [0,\infty) \to \mathbb{R}^{d+1}$ as
\begin{equation}
J(x,t) = (x, H(x,t)).
\end{equation}
The metric tensor of $\mathcal{S}(t)$ in local coordinates $x \in \mathbb{R}^d$ is given by
\begin{equation*}
G(x,t) = I + \nabla H(x,t) \otimes \nabla H(x,t)
\end{equation*}
and we define 
\begin{equation*}
|G|(x,t) := \det G(x,t) = 1+ |\nabla H(x,t)|^2.
\end{equation*}
Due to the physical application of the hypersurface, we restrict the dimension to $d=2$, but mathematically all results below also apply for general dimension $d$. 
The dimension would become relevant if one considered taking the cutoff away, as the surface becomes rougher with increasing dimension.\\
Motivated by the Helfrich elastic fluctuating membrane model presented  in the introduction \eqref{membrane}, we  assume that the random field $H(x,t)$ is Gaussian and it can be written as $H(x, t) = h(x,\eta_{t})$ via Fourier expansion with ultra-violet cutoff (for fixed $\tilde{K}\in\N$). More precisely, we define
\begin{align}\label{eq:h}
H(x,t):=h(x,\eta_{t}):=\sum_{k\in\mathbb{Z}^{2}:\,0<\abs{k}\leqslant \tilde{K}}\eta^{k}_{t}e_{k}(x),\quad x\in\mathbb{T}^{2}
\end{align}
(without the $k=0$ Fourier mode) where $(e_{k})_{k\in\mathbb{Z}^{2}}$ is the Fourier basis on the two-dimensional torus $\mathbb{T}^{2}$, i.e. $e_{k}(x)=\exp(2\pi i k\cdot x) \in C^\infty(\mathbb{T}^2)$ and $\mathbb{T}^{2}$ is identified with $[0,1]^{2}$. The cutoff is needed to ensure differentiability of the surface. The equation \eqref{membrane} for $H$ then yields the dynamics of $\eta=(\eta^{k})_{k\in\mathbb{Z}^{2}:0<\abs{k}\leqslant \tilde{K}}$, which is given by the complex-valued Ornstein Uhlenbeck processes with
\begin{align*}
d\eta_{t}=-\Gamma\eta_{t}dt+\sqrt{2\Gamma\Pi}dW_{t}
\end{align*} 
for a $K$-dimensional standard complex-valued Brownian motion $W$  with complex conjugate $\overline{W^{k}_{t}}=W^{-k}_{t}$, where $K:=\#\{k\in\mathbb{Z}^{2}\mid 0<\abs{k}\leqslant\tilde{K}\}$ and where $\Gamma,\Pi$ are real matrices, such that $\overline{\eta^{k}}=\eta^{-k}$, which implies that $H$ is real-valued as $\overline{H}=H$. 
Since we work in the real-valued setting, we
identify $\eta$ with $(Re(\eta),Im(\eta))$, which is a $2K$-dimensional real-valued Ornstein-Uhlenbeck process with the above dynamics for a $2K$-dimensional real-valued standard Brownian motion $W$ and with the property that $Re(\eta^{k})=Re(\eta^{-k})$ and $Im(\eta^{k})=-Im(\eta^{-k})$. The drift matrix $\Gamma$ is symmetric and positive definite, defined by $\Gamma:=diag(\Gamma_{k})$ with 
\begin{align*}
\Gamma_{k}=\frac{\kappa^{\ast}\abs{2\pi k}^{4}+\sigma^{\ast}\abs{2\pi k}^{2}}{\abs{2\pi k}},
\end{align*}
where $\kappa^{\ast}=\kappa/(k_{B}T),\sigma^{\ast}=\sigma/(k_{B}T)$ are positive constants that depend on the geometry of $\mathcal{S}$.
The diffusion matrix $\Pi$ is also symmetric and positive definite and defined as $\Pi:=diag(\Pi_{k})$ with 
\begin{align*}
\Pi_{k}=\frac{1}{\kappa^{\ast}\abs{2\pi k}^{4}+\sigma^{\ast}\abs{2\pi k}^{2}}.
\end{align*}
 
\noindent Since the symmetric, positive-definite matrices $\Gamma,\Pi$ commute, the normal distribution
\begin{equation}
 N(0,\Pi)=:\rho_{\eta}
 \end{equation}
is the invariant measure for the Ornstein Uhlenbeck process $\eta$. Then we have 
\begin{equation}\label{rho_eta}
   \rho_{\eta}(d\eta)=\rho_{\eta}d\eta=\frac{1}{\sqrt{(2\pi)^{2}\abs{\Pi}}}\exp(-\frac{1}{2}\eta\cdot\Pi^{-1}\cdot\eta)d\eta,
\end{equation} with $\abs{\Pi}:=\det(\Pi)$ (note that we use the same notation $\rho_{\eta}$ for the measure and its density).
The generator $\mathcal{L}_{\eta}$ of $\eta$ is given by 
\begin{equation}\label{OUgen}
 \mathcal{L}_{\eta}=-\Gamma\eta\cdot\nabla_{\eta}+\Pi\Gamma:\nabla_{\eta}\nabla_{\eta}.
\end{equation}
Here and later on,  we will use the notation
 $$A:\nabla_{x}\nabla_{x}f(x):=\sum_{i,j=1}^{n}A_{i,j}\partial_{x_{i}}\partial_{x_{j}}f(x) \quad \text{for } A\in\R^{n\times n}, f\in C^{2}(\R^{n},\R), n\in\N.$$\\
The generator $\mathcal{L}_{\eta}$ is a closed, unbounded operator on $L^{2}(\rho_{\eta})$ with domain $dom(\mathcal{L}_{\eta})=\{f\in L^{2}(\rho_{\eta})\mid \mathcal{L}_{\eta}f\in L^{2}(\rho_{\eta})\}$. Observe that, since  $\mathcal{L}_{\eta}^{\ast}\rho_{\eta}=0$ for the Lebesgue adjoint $\mathcal{L}_{\eta}^{\ast}$, the invariance of $\rho_{\eta}$, i.e. 
 $$\langle\mathcal{L}_{\eta}f\rangle_{\rho_{\eta}}=\int \mathcal{L}_{\eta}f(x)\rho_{\eta}(x)dx=0, \quad  \text{for all } f\in dom(\mathcal{L}_{\eta}),$$
  can  be checked easily. Let us define the space
 \begin{equation}
 H^{1}(\rho_{\eta}):=\{f\in dom(\mathcal{L}_{\eta})\mid \langle (-\mathcal{L}_{\eta})f,f\rangle_{\rho_{\eta}}<\infty\}.
 \end{equation}
Furthermore, notice that for the Ornstein-Uhlenbeck generator $\mathcal{L}_{\eta}$, spectral gap estimates hold true, that is, there exists a constant $C>0$, such that 
\begin{align}\label{eq:sp-gap-eta}
\norm{f}_{H^{1}(\rho_{\eta})}^{2}=\langle(-\mathcal{L}_{\eta}) f,f\rangle_{\rho_{\eta}}\geqslant C\norm{f-\langle f\rangle_{\rho_{\eta}}}_{\rho_{\eta}}^{2}
\end{align} for any $f\in H^{1}(\rho_{\eta})$ with $\langle f\rangle_{\rho_{\eta}}:=\int f(\eta)d\rho_{\eta}$. Indeed, from a simple calculation using the  invariance  for $f^{2}$, it follows that $\langle(-\mathcal{L}_{\eta}) f,f\rangle_{\rho_{\eta}}=\langle \Pi\Gamma\nabla_{\eta} f,\nabla_{\eta} f\rangle_{\rho_{\eta}}$. Then \eqref{eq:sp-gap-eta} follows from $\min_{\mathbb{T}^{2}}\rho_{\eta}>0$ and the Poincaré-inequality for the Laplacian on $H^{1}(\mathbb{T}^{2},d\eta)$.\\ 
In particular, $\rho_{\eta}$ is an ergodic measure for the Ornstein-Uhlenbeck process $\eta$. Furthermore, if $f$ is centered under $\rho_{\eta}$, then $P_{t}^{\eta}f$ is centered by invariance and thus \eqref{eq:sp-gap-eta} applied for $P_{t}^{\eta}f$ together with $\partial_{t}\langle P_{t}^{\eta}f,P_{t}^{\eta}f\rangle_{\rho_{\eta}}=2\langle\mathcal{L}_{\eta}P_{t}^{\eta}f,P_{t}^{\eta}f\rangle_{\rho_{\eta}}$ yield the spectral gap estimates for the semigroup $(P_{t}^{\eta})_{t\geqslant 0}$ of the Ornstein-Uhlenbeck process:
\begin{align*}
\norm{P_{t}^{\eta}f-\langle f\rangle_{\rho_{\eta}}}_{L^{2}(\rho_{\eta})}\leqslant e^{-Ct}\norm{f}_{L^{2}(\rho_{\eta})}
\end{align*} for all $t\geqslant 0$ and $f\in L^{2}(\rho_{\eta})$ (sometimes also called exponential ergodicity). 

We want to consider a Brownian motion $X$ on the Helfrich membrane $H$ given in \eqref{eq:h}. This diffusion will be driven by an independent Brownian motion $B$. For each fixed realization of the membrane, 
 it will be the Markov process that has in local coordinates the Laplace-Beltrami operator $\mathcal{L}^{H}=\mathcal{L}$ as a generator. As we assume the expansion \eqref{eq:h} with coefficients $\eta$, we obtain a system of SDEs for $(X,\eta)$ describing the dynamics of the diffusion $X$ on the membrane $H$. As in \cite{dep}, we thus define:

\begin{definition}
Let $(\Omega,\F,\p)$ be a probability space and $B$ a two-dimensional standard Brownian motion independent of a $2K$-dimensional standard Brownian motion $W$. Let $x_{0}$ be a random variable with values in $\mathbb{T}^{2}$ independent of $B$ and $W$. Let $(X,\eta)$ be the solution of the following system of SDEs
\begin{align}\label{eq:system}
&dX_{t}=F(X_{t},\eta_{t})dt+\sqrt{2\Sigma(X_{t},\eta_{t})}dB_{t},\quad X_{0}=x_{0}\\
&d\eta_{t}=-\Gamma\eta_{t}dt+\sqrt{2\Gamma\Pi}dW_{t},\quad \eta_{0}\sim\rho_{\eta}
\end{align} 
with $\Sigma:\mathbb{T}^{2}\times\R^{2K}\to\R^{2\times 2}_{sym}$, where $\Sigma(x,\eta)$ is the inverse of the metric tensor matrix $g(x,\eta)\in\R^{2\times 2}$  defined as 
\begin{align}\label{eq:defSigma}
\Sigma(x,\eta):=g^{-1}(x,\eta):=(I+\nabla_{x} h(x,\eta)\otimes \nabla_{x} h(x,\eta))^{-1}
\end{align} and $F:\mathbb{T}^{2}\times\R^{2K}\to\mathbb{T}^{2}$ with (notation: $\abs{g}=\abs{\Sigma^{-1}}:=\det(\Sigma^{-1})$)
\begin{align}\label{eq:defF}
F(x,\eta):=\frac{1}{\sqrt{\abs{g}(x,\eta)}}\nabla_{x}\cdot (\sqrt{\abs{g}}g^{-1}(x,\eta)).
\end{align}
Then we call $X$ a Brownian motion on the Helfrich membrane $H$ (started in $x_{0}$).
\end{definition}
\noindent 
One can show  (cf. \cite[Proposition 2.3.1]{duncan}) that the solution $(X,\eta)$ exists and it is a Markov process with generator on smooth, compactly supported test functions $f:\mathbb{T}^{2}\times\R^{2K}\to\R$ given by 
\begin{align*}
(\mathcal{L}+\mathcal{L}_{\eta})f(x,\eta)&=\frac{1}{\sqrt{\abs{g}(x,\eta)}}\nabla_{x}\cdot (\sqrt{\abs{g}(x,\eta)}\Sigma(x,\eta)\nabla_{x}f(x,\eta))+\mathcal{L}_{\eta}f(x,\eta)\\&=F(x,\eta)\nabla_{x} f(x,\eta)+\Sigma (x,\eta):\nabla_{x}\nabla_{x}f(x,\eta)+\mathcal{L}_{\eta}f(x,\eta),
\end{align*} 
where $\mathcal{L}_{\eta}$ is the generator of the Ornstein Uhlenbeck process $\eta$ given with (\ref{OUgen}). Moreover, from the proof of \cite[Proposition 2.3.1]{duncan}, we conclude the following uniform bounds: there exists a constant $C_{1}>0$ such that
\begin{align}\label{eq:Sigma}
\abs{\Sigma(x,\eta)}_{F}\leqslant C_{1},\qquad\forall (x,\eta)\in\mathbb{T}^{2}\times\R^{2K},
\end{align} where $\abs{\cdot}_{F}$ denotes the Frobenius-norm and there exists a constant $C_{2}>0$ such that 
\begin{align}\label{eq:F}
\abs{F(x,\eta)}\leqslant C_{2}(1+\abs{\eta}),\qquad\forall (x,\eta)\in\mathbb{T}^{2}\times\R^{2K},
\end{align} 
where $\abs{\cdot}$ denotes the usual euclidean norm.

\noindent We are interested in considering fluctuations of the membrane in time ($\epsilon^{\beta}$, $\beta\geqslant 0$ or $\beta =-\infty$) and space ($\epsilon^{\alpha}$, $\alpha\geqslant 0$) with different speeds $\alpha,\beta$. More precisely, we consider instead of $H(x,t)$ the fluctuating surface $\epsilon^{\alpha}H(\frac{x}{\epsilon^{\alpha}},\frac{t}{\epsilon^{\beta}})=\epsilon^{\alpha}h(\epsilon^{-\alpha}x,\eta_{\epsilon^{-\beta}t})$, which transforms the system of equations (\ref{eq:system}) into
\begin{align}\label{eq:rsystem}
&dX^{\epsilon}_{t}=\frac{1}{\epsilon^{\alpha}}F\paren[\bigg]{\frac{X^{\epsilon}_{t}}{\epsilon^{\alpha}},\eta^{\epsilon}_{t}}dt+\sqrt{2\Sigma\paren[\bigg]{\frac{X^{\epsilon}_{t}}{\epsilon^{\alpha}},\eta^{\epsilon}_{t}}}dB_{t} \\
&d\eta^{\epsilon}_{t}=-\frac{1}{\epsilon^{\beta}}\Gamma\eta^{\epsilon}_{t}dt+\sqrt{\frac{2\Gamma\Pi}{\epsilon^{\beta}}}d\tilde{W}_{t},\quad\eta_{0}^{\epsilon}\sim\rho_{\eta}.
\end{align} 
Here we define the Brownian motion $\tilde{W}_{t}:=\epsilon^{\beta/2}W_{\epsilon^{-\beta}t}$ and thus $(\eta_{\epsilon^{-\beta}t})_{t\geqslant 0}=(\eta^{\epsilon}_{t})_{t\geqslant 0}$. Since we are interested in convergence in distribution, we may replace $\tilde{W}$ by $W$, to abuse the notation.
We refer to \cite[section 2.3.3]{duncan} for the derivation of the system and the physical background.
Furthermore, note that stationarity and Gaussianity imply boundedness of all moments of $\eta^{\epsilon}_{t}$ in $\epsilon, t$. 

As already indicated, we want to study the behaviour of the Itô and Stratonovich rough path lift of $(X^{\epsilon})$ for different speeds $\alpha,\beta$ when $\epsilon\to 0$.
More precisely, by relabeling $\epsilon^{\alpha}$, all relevant regimes to consider are $(\alpha,\beta)=(0,1)$ and $(\alpha,\beta)\in {1}\times [-\infty,\infty)$ (i.e. all other scaling regimes for $\alpha$ can be transformed into one of these, yielding the same limit behavior). By the considerations in \cite[section 6]{duncan}, it moreover turns out that the regimes $(\alpha,\beta)\in\{(0,1),(1,2),(1,1)\}$ are the most interesting ones, in the sense that they yield, together with the quenched regime $(\alpha,\beta)=(1,-\infty)$, the four different limit behaviours that can occur (cf. \cite[Theorem 6.0.1]{duncan}).  
There is one scaling regime, $\alpha=1$ and $\beta\in (2,3]$, where the limit is actually open, as certain Poisson equations needed to prove the limit might not have solutions. 
Considering the regimes $(\alpha,\beta)\in\{(0,1),(1,2),(1,1)\}$, we prove the   convergence of the rough path lift and identify the limit. In the quenched regime, $(\alpha,\beta)=(1,-\infty)$ with $\eta_{0}$ being deterministic, the limit follows from the work \cite{radditive}, cf. the rough central limit theorem for diffusions with periodic coefficients in \cite[section 4.3]{radditive}.\\  
The convergence results will be proven in $\gamma$-Hölder rough path topology. Here we briefly recall the definition of a $\gamma$-Hölder rough path and for more details we refer the reader to \cite{fh}. We will write $X_{s,t} := X_t - X_s$ and we define 
$$\Delta_{T}:=\{(s,t)\in[0,T]^{2}\mid s\leqslant t\} .$$

\begin{definition}\cite[Def. 2.1]{fh}
For $\gamma\in (1/3,1/2]$ we call $(X,\mathbb{X})\in C([0,T],\R^{d})\times C(\Delta_{T},\R^{d\times d})$ a $\gamma$-Hölder rough path if:
\begin{enumerate}
\item[i)] Chen's relation holds, that is 
\begin{align*}
\mathbb{X}_{r,t}-\mathbb{X}_{r,s}-\mathbb{X}_{s,t}=X_{s,u}\otimes X_{u,t}
\end{align*} for all $0\leqslant r\leqslant s\leqslant t\leqslant T$ with $\mathbb{X}_{t,t}=0$, where $X_{s,t}:=X_{t}-X_{s}$,

\item[ii)] the (inhomogeneous) $\gamma$-Hölder norms are finite, that is
\begin{align*}
\norm{(X,\mathbb{X})}_{\gamma}:=\norm{X}_{\gamma,T}+\norm{\mathbb{X}}_{2\gamma,T}:=\sup_{0\leqslant s<t\leqslant T}\frac{\abs{X_{s,t}}}{\abs{t-s}^{\gamma}}+\sup_{0\leqslant s<t\leqslant T}\frac{\abs{\mathbb{X}_{s,t}}}{\abs{t-s}^{2\gamma}}<\infty.
\end{align*}
\end{enumerate}  
We denote the nonlinear space of all such $\gamma$-Hölder rough paths by $C_{\gamma,T}$ equipped with the distance
\begin{align*}
\norm{(X^{1},\mathbb{X}^{1});(X^{2},\mathbb{X}^{2}) }_{\gamma}=\norm{X^1-X^2}_{\gamma,T}+\norm{\mathbb{X}^1-\mathbb{X}^2}_{2\gamma,T}.
\end{align*}
\end{definition}

\begin{remark}
For a two-dimensional Brownian motion $B$, the Itô lift $(B,\mathbb{B}_{\mathrm{Ito}})$, where $\mathbb{B}_{\mathrm{Ito}}(s,t):=\int_{s}^{t}B_{s,r}\otimes dB_{r}$ are Itô-integrals, as well as the Stratonovich lift $(B,\mathbb{B}_{\mathrm{Strato}})$ for $\mathbb{B}_{\mathrm{Strato}}(s,t):=\int_{s}^{t}B_{s,r}\otimes\circ dB_{r}$ being Stratonovich integrals, are for any $\epsilon>0$ almost surely $\gamma$-Hölder rough path for $\gamma=1/2-\epsilon$, cf. \cite[Ch. 3]{fh}. But also $(B,(s,t)\mapsto\mathbb{B}_{s,t}+A(t-s))$ is a $\gamma$-rough path, where $A$ is a matrix $A\in\R^{2\times 2}$ and $\mathbb{B}=\mathbb{B}_{\mathrm{Ito}}$ or $\mathbb{B}=\mathbb{B}_{\mathrm{Strato}}$. The latter will be the lift of the Brownian motion that we encounter below.
\end{remark}

We finalize this section by recalling the concept of uniform controlled variations by Kurtz and Protter (\cite[Definition 7.3]{kp}; here for continuous semimartingales without the need of stopping times). 

\begin{definition}
A sequence $(X^{\epsilon})_{\epsilon}$ of $\R^{d}$-valued continuous semi-martingales on $[0,T]$ decomposed as the sum $X^{\epsilon} = M^{\epsilon} +A^{\epsilon}$, where $M^{\epsilon}$ is a local martingale  and $A^{\epsilon}$ is of finite variation,  satisfies the UCV (Uniformly Controlled Variations) condition if and only if
\begin{equation}\label{def:UCV}
( \left< M^{\epsilon,i} \right>_T)_{\epsilon} \quad \text{and} \quad (\mathrm{Var}_{1,[0,T]} (A^{\epsilon}))_{\epsilon} \quad \text{are tight in } \R,
\end{equation}
for $i=1,...,d$ and $\mathrm{Var}_{1,[0,T]}(f):=\lim_{\abs{\pi}\to 0}\sum_{s,t\in\pi}\abs{f_{t}-f_{s}}$ denotes the one-variation of a function $f:[0,T]\to \R^{d}$ and the limit is taken over all finite partitions $\pi$ of $[0,T]$ with mesh size $\abs{\pi}=\max_{s,t\in\pi}\abs{t-s}\to 0$.
\end{definition}
\begin{remark}
The tightness \eqref{def:UCV} follows from
\begin{align}\label{UCVcond}
\max_{i=1,\dots,d}\sup_{\epsilon}\paren[\big]{\E[\langle M^{\epsilon,i}\rangle _{T}]+\E[\mathrm{Var}_{1,[0,T]}(A^{\epsilon})]}<\infty,
\end{align}
which is the bound, as we will verify when showing that a semimartingale satisfies the UCV condition. 
\end{remark}

We state a version of \cite[Definition 7.1, Theorem 7.4]{kp}, that will be repeatedly exploited in the sequel in order to prove the distributional convergence of certain Itô integrals. For the proof see \cite[Thm. 7.4]{kurtz1996weak} and \cite[Thm. 2.2]{kp}.

\begin{proposition}{\cite[Thm. 7.4]{kurtz1996weak}}\label{prop:UCV}
A sequence $(X^{\epsilon})_{\epsilon}$ of $\R^{d}$-valued continuous semi-martingales on $[0,T]$ satisfies the UCV condition if and only if for all sequences $(Y^{\epsilon})_{\epsilon}$ with $(Y^{\epsilon},X^{\epsilon})\Rightarrow (Y,X)$ jointly in distribution in $C([0,T],\R^{2d})$ and with $Y^{\epsilon}$ integrable against $X^{\epsilon}$ and $Y$ against $X$ in the Itô sense, it follows that $(Y^{\epsilon},X^{\epsilon},\int_{0}^{\cdot} Y^{\epsilon}_{s}\otimes dX^{\epsilon}_{s})\Rightarrow (Y,X,\int_{0}^{\cdot} Y_{s} \otimes dX_{s})$ in distribution in $C([0,T],\R^{2d+d\times d})$.
\end{proposition}

Let us furthermore state a version of the Kolmogorov criterion for rough paths, \cite[Theorem 3.1]{fh}, that we use in the sequel.
\begin{lemma}\label{lem:rp-kol}
Let $(X^{\epsilon},\mathbb{X}^{\epsilon})_{\epsilon}$ be a family of $\gamma$-Hölder rough path, $\gamma<1/2$. Assume that for any $p> 2$, there exist constants $C_{1},C_{2}>0$ such that:
\begin{align}\label{eq:kx}
\sup_{\epsilon}\E[\abs{X^{\epsilon,i}_t-X^{\epsilon,i}_s}^{p}]\leqslant C_{1}\abs{t-s}^{p/2},\quad\forall s,t\in[0,T]
\end{align} and
\begin{align}\label{eq:kxx}
\sup_{\epsilon}\E\bigg[\abs[\bigg]{\int_{s}^{t}X^{\epsilon, i}_{s,r} dX^{\epsilon,j}_{r}}^{p/2}\bigg]\leqslant C_{2}\abs{t-s}^{p/2},\quad\forall s,t\in\Delta_{T}
\end{align} for $i,j\in\{1,...,d\}$.
Then for any $\gamma'<1/2$,
\begin{align*}
\sup_{\epsilon}\E[\norm{(X^{\epsilon},\mathbb{X}^{\epsilon})}_{\gamma'}^{p}]<\infty.
\end{align*}
If furthermore $\sup_{\epsilon}\E[\abs{X_{0}^{\epsilon}}]<\infty$ (i.e. tightness of $(X_{0}^{\epsilon})_{\epsilon}$) holds true, then it follows that $(X^{\epsilon},\mathbb{X}^{\epsilon})_{\epsilon}$ is tight in $C_{\gamma,T}$. 
\end{lemma}
\begin{proof}
The proof follows from the proof of \cite[Theorem 3.1]{fh} applied to $\beta=1/2$ and $\gamma'=\beta-1/q-\epsilon$, $\epsilon>0$, using the fact that, according to assumption \eqref{eq:kx}, \eqref{eq:kxx} holds for any $p>2$.
Tightness in $C_{\gamma,T}$ then follows utilizing the compact embedding $C_{\gamma', T}\hookrightarrow C_{\gamma, T}$ for $\gamma<\gamma'$.
\end{proof}

\end{section}

\begin{section}{Membrane with purely temporal fluctuations}\label{sect:averaging}

In this section, we consider the scaling regime $\alpha=0$, $\beta=1$ in \eqref{eq:rsystem} and thus obtain the slow-fast system:
 \begin{align}
 dX^{\epsilon}_{t} & = F(X^{\epsilon}_{t}, \eta^{\epsilon}_{t}) dt + \sqrt{2 \Sigma (X^{\epsilon}_{t}, \eta^{\epsilon}_{t})} dB(t) \\
 d \eta^{\epsilon}_{t} & = - \frac{1}{\epsilon} \Gamma \eta^{\epsilon}_{t} dt + \sqrt{\frac{2}{\epsilon} \Gamma \Pi} dW_{t},\quad \eta^{\epsilon}_{0}\sim\rho_{\eta},
  \end{align}
where $B$ and $W$ are independent Brownian motions.

From classical stochastic averaging, see also \cite[Theorem 4]{dep}, we know that $X^\epsilon \Rightarrow X $ in distribution in $C([0,T], \mathbb{R}^2)$ as $\epsilon\to 0$. The limit $X$ is the solution of the averaged system
\[
dX_t = \overline{F} (X_t) dt + \sqrt{2 \overline{\Sigma}(X_t)} dB_t,
\]
with
\begin{align}
\overline{F}(x) &:= \int_{\mathbb{R}^{2K}} F(x,\eta) \rho_\eta (d \eta), \\
\overline{\Sigma}(x) &:= \int_{\mathbb{R}^{2K}} \Sigma(x,\eta) \rho_\eta(d\eta) ,
\end{align}
where $\rho_{\eta}(d\eta)$ is the invariant measure of $\eta$ given by (\ref{rho_eta}).
Utilizing the linear growth of $F$ in $\eta$ uniformly in $x$, cf. \eqref{eq:F},  boundedness of all moments of $\eta^{\epsilon}_{t}$ in $\epsilon, t$  and boundedness of $\Sigma$, cf. \eqref{eq:Sigma}, we can conclude that $(X^{\epsilon})_{\epsilon}$ satisfies the UCV condition (\ref{UCVcond}); see below for the detailed proof. Thus, according to Proposition \ref{prop:UCV}, the iterated Itô integrals of $X^{\epsilon}$ will converge to the iterated Itô integrals of the limit $X$. More precisely, the following theorem holds:
\begin{theorem} \label{thm:ucvlimit}
Let $\gamma< 1/2$  and $X^{\epsilon}$ and $X$ be as above. Let $\mathbb{X}^{\epsilon}_{s,t}:=\int_{s}^{t} (X^{\epsilon}_r-X^{\epsilon}_s)\otimes dX^{\epsilon}_r$ and $\mathbb{X}_{s,t}:=\int_{s}^{t} (X_r-X_s)\otimes dX_r$, where the stochastic integrals are understood in the Itô sense. Then it follows that
\begin{align}
(X^{\epsilon},\mathbb{X}^{\epsilon})\Rightarrow (X,\mathbb{X})
\end{align} in distribution in $\gamma$-Hölder rough path topology.
\end{theorem}
\begin{proof}
We first prove the weak convergence of the iterated integrals $(\mathbb{X}^{\epsilon}_{0,t})$ in $\R^{2\times 2}$ for any $t\geqslant 0$ and then show tightness of $(X^{\epsilon},\mathbb{X}^{\epsilon})$ in $\gamma$-Hölder rough path topology.\\ 
The first aim is to apply \cref{prop:UCV} and show that the sequence $(X^{\epsilon})_{\epsilon}$ satisfies the UCV condition. We have that $X^{\epsilon}=A^{\epsilon}+M^{\epsilon}$ for 
\begin{align*}
A^\epsilon_t &:= \int_0^t  F(X^\epsilon_s, \eta^\epsilon_s)ds \\
M^\epsilon_t &:=  \int_0^t  \sqrt{2 \Sigma (X^{\epsilon}_s, \eta^\epsilon_s )} dB_s,
\end{align*} where $A^{\epsilon}$ is of finite variation and $M^{\epsilon}$ is a martingale. For $(X^{\epsilon})_{\epsilon}$ to satisfy the UCV condition, we thus have to show the bound (\ref{UCVcond}).

By boundedness of $\Sigma$ it is immediate that the expected quadratic variation of $M^{\epsilon}$ is also uniformly bounded in $\epsilon$. For the bound on the total variation of $A^{\epsilon}$ we use that (\ref{eq:F}) holds uniformly in $x\in\mathbb{T}^{2}$, such that:
\begin{align*}
\sup_\epsilon \E\big[ \mathrm{Var}_{1,[0,T]}(A^{\epsilon}) \big] &\leq  
\sup_\epsilon \E\bigg[ \int_0^T |F(X^\epsilon_s, \eta^\epsilon_s)| ds\bigg] \\
&\leq C\sup_\epsilon \E\bigg[\int_0^T (1 + |\eta^{\epsilon}_s|) ds\bigg] 
\\&= C(1+\E[\abs{\eta_{0}}]) T,
\end{align*}
where in the last equality, we used the stationarity of $\eta^{\epsilon}$. As we have weak convergence of $(X^{\epsilon})_\epsilon$ to $X$ in $C([0,T],\R^{2})$, this implies, according to  \cref{prop:UCV}, that we also have weak convergence of the Itô integrals 
$\mathbb{X}^{\epsilon}:=\int X^{\epsilon}\otimes dX^{\epsilon}$ to $\mathbb{X}:=\int X\otimes dX$ in $C(\Delta_{T},\R^{2\times 2})$.\\
To prove tightness in $\gamma$-Hölder rough path topology for $\gamma<1/2$, we utilize \cref{lem:rp-kol}.
Here \eqref{eq:kx} follows immediately from the linear growth of $F$ in $\eta$ and bounded moments of $\eta^{\epsilon}_{t}$ in $\epsilon,t$ by stationarity, as well as Burkholder-Davis-Gundy inequality for the martingale part and boundedness of $\Sigma$. Moreover, \eqref{eq:kxx} follows from \eqref{eq:kx} and the estimate
\begin{align*}
\E\bigg[\abs[\bigg]{\int_{s}^{t}X^{\epsilon, i}_{s,r}dX^{\epsilon, j}_{r}}^{p/2}\bigg]
&\lesssim
\E\bigg[\abs[\bigg]{\int_{s}^{t}X^{\epsilon, i}_{s,r}dM^{\epsilon, j}_{r}}^{p/2}\bigg]+\E\bigg[\abs[\bigg]{\int_{s}^{t}X^{\epsilon, i}_{s,r}dA^{\epsilon, j}_{r}}^{p/2}\bigg]\\
&\lesssim \E\bigg[\paren[\bigg]{\int_{s}^{t}\abs{X^{\epsilon, i}_{s,r}}^{2}dr}^{p/4}\bigg]+
\E\bigg[\paren[\bigg]{\int_{s}^{t}\abs{X^{\epsilon, i}_{s,r}F^{j}(X^{\epsilon}_{r},\eta^{\epsilon}_{r})}dr}^{p/2}\bigg]\\
&\lesssim \paren[\bigg]{\int_{s}^{t}\E[\abs{X_{s,r}^{\epsilon,i}}^{p/2}]^{4/p}dr}^{p/4}+\paren[\bigg]{\int_{s}^{t}\E[\abs{X_{s,r}^{\epsilon,i}F^{j}(X_{r}^{\epsilon},\eta_{r}^{\epsilon})}^{p/2}]^{2/p}dr}^{p/2}\\
&\lesssim\abs{t-s}^{(2\times\frac{1}{2}+1)\times\frac{p}{4}}+\paren[\bigg]{\int_{s}^{t}\E[\abs{X_{s,r}^{\epsilon,i}}^{p}]^{1/p}\E[\abs{F^{j}(X_{r}^{\epsilon},\eta_{r}^{\epsilon})}^{p}]^{1/p}dr}^{p/2}\\
&\lesssim \abs{t-s}^{p/2}+\paren[\bigg]{\int_{s}^{t}\E[\abs{X_{s,r}^{\epsilon,i}}^{p}]^{1/p}\E[(1+\abs{\eta_{r}^{\epsilon}})^{p}]^{1/p}dr}^{p/2}\\
&\lesssim \abs{t-s}^{p/2}+\paren[\bigg]{\int_{s}^{t}\E[\abs{X_{s,r}^{\epsilon,i}}^{p}]^{1/p}dr}^{p/2}\\
&\lesssim \abs{t-s}^{p/2}+\abs{t-s}^{p/4+p/2}\lesssim_{T}\abs{t-s}^{p/2}
\end{align*} 
using the Burkholder-Davis-Gundy inequality for the martingale part, boundedness of $\Sigma$ in the second line and the generalized Minkowski's inequality for integrals for both summands in the third line (and the linear growth of $F$ and stationarity of $\eta$).\\ 
\noindent Combining distributional convergence of $(X^{\epsilon},\mathbb{X}^{\epsilon})_\epsilon$ to $(X,\mathbb{X})$ in $C([0,T],\R^{d})\times C(\Delta_{T},\R^{d\times d})$ and tightness in $C_{\gamma,T}$, we conclude on distributional convergence in $\gamma$-Hölder rough path topology for $\gamma<1/2$.
\end{proof}
\end{section}

\begin{section}{Membrane with temporal fluctuations twice as fast as spatial fluctuations}\label{sect:hom1}

In this section, we consider the scaling regime $\alpha=1$, $\beta=2$ in \eqref{eq:rsystem}, that is, temporal fluctuations occur twice as fast as spatial ones. We introduce the fast process 
\begin{align*}
Y^\epsilon_t:= \frac{X^\epsilon_t}{\epsilon} \mod \mathbb{T}^2.
\end{align*} 
Then the general SDE system can be written as 

 \begin{subequations}\label{eq:xye}
\begin{empheq}[left=\empheqlbrace]{align}
 dX^\epsilon_t & = \frac{1}{\epsilon}F(Y^\epsilon_t, \eta^\epsilon_t ) dt + \sqrt{2 \Sigma (Y^\epsilon_t, \eta^\epsilon_t )} dB_t,  \\
 dY^\epsilon_t & = \frac{1}{\epsilon^2}F(Y^\epsilon_t, \eta^\epsilon_t ) dt + \sqrt{\frac{2}{\epsilon^2} \Sigma (Y^\epsilon_t, \eta^\epsilon_t )} dB_t, \\
 d \eta^\epsilon_t & = - \frac{1}{\epsilon^2} \Gamma \eta^\epsilon_t dt + \sqrt{\frac{2}{\epsilon^2} \Gamma \Pi} dW_t, 
 \end{empheq}
\end{subequations}
for independent Brownian motions $B$ and $W$, where $B$ is a two-dimensional and $W$ is a $2K$-dimensional standard Brownian motion.\\
Utilizing Itô's formula, one can easily check that on smooth, compactly supported functions $f\in C^{\infty}_{c}(\mathbb{T}^{2}\times\R^{2K},\R)$, the infinitesimal generator of the fast process $(Y^\epsilon, \eta^\epsilon)$ is $\epsilon^{-2} \mathcal{G}$, where 
\begin{align}
\mathcal{G} = \mathcal{L}_0 + \mathcal{L}_\eta
\end{align} for
\begin{align}\label{eq:L0}
\mathcal{L}_{0}f(y,\eta)=F(y,\eta)\cdot\nabla_{y} f(y,\eta)+\Sigma (y,\eta):\nabla_{y}\nabla_{y}f(y,\eta),
\end{align} 
which is the generator of $Y$ (for fixed $\eta$) and $\mathcal{L}_{\eta}$ 
 is the generator of the Ornstein-Uhlenbeck process $\eta$, given by (\ref{OUgen}), cf. also \cite[section 5.3]{duncan}. We also write $\mathcal{L}_{0}(\eta)$ to denote the operator $\mathcal{L}_{0}$ acting on functions $f:\mathbb{T}^{2}\to\R$, stressing the dependence on fixed $\eta\in\R^{2K}$. The reason for introducing the fast process $Y^{\epsilon}$ is that the drift term of $X^{\epsilon}$ is given as an (unbounded) additive functional of the Markov process $(Y^{\epsilon},\eta^{\epsilon})$. Moreover, we have the equality in law, $(Y_t,\eta_{t})_{t\geqslant 0}\stackrel{d}{=}(Y^{\epsilon}_{\epsilon^{2} t},\eta^{\epsilon}_{\epsilon^{2} t})_{t\geqslant 0}$, where $(Y,\eta)$ is the Markov process with generator $\mathcal{G}$. 
 
One can show (cf. \cite[Prop. 5.3.1]{duncan}) that there exists a unique invariant measure $\rho$ for the Markov process $(Y,\eta)$, whose density is the unique, normalized solution of 
\begin{equation}\label{eq:rho}
\mathcal{G}^* \rho = 0,
\end{equation}
$\mathcal{G}^{\ast}$ being the adjoint operator of $\mathcal{G}$ with respect to $L^{2}(dyd\eta)$. As $\rho$ is the unique invariant measure, it is in particular ergodic for $(Y,\eta)$. Furthermore, we can extend the semigroup $(P^{(Y,\eta)}_{t})_{t\geqslant 0}$ of the Markov process $(Y,\eta)$, with $P_{t}^{(Y,\eta)}f(y,\eta)=\E[f(Y_{t},\eta_{t})\mid (Y_{0},\eta_{0})=(y,\eta)]$ for $f\in C^{\infty}_{c}(\mathbb{T}^{d}\times\R^{2K})$, uniquely to a strongly continuous contraction semigroup on $L^{2}(\rho)$ (that is possible by invariance of $\rho$, cf. \cite[Theorem 1, p. 381]{Yosida}) and define the generator $\mathcal{G}: dom(\mathcal{G})\subset L^{2}(\rho)\to L^{2}(\rho)$ with $dom(\mathcal{G})=\{u\in L^{2}(\rho)\mid\mathcal{G}u\in L^{2}(\rho)\}$ and $\mathcal{G}u:=\lim_{t\to 0}t^{-1}(P_{t}^{(Y,\eta)}u-u)$ with limit in $L^{2}(\rho)$.\\
Let us define
\begin{align*}
V(\eta):=1+\frac{1}{2}\abs{\eta}^{2}.
\end{align*}

\noindent Then, according to  the proof of \cite[Prop. 5.3.1]{duncan}, $V$ is a Lyapunov function for the fast process $(Y^{\epsilon},\eta^{\epsilon})$ and we have the pointwise spectral-gap-type estimates of the form
\begin{equation}\label{eq:p-sp}
 \abs{P_{t}^{(Y,\eta)}f(y,\eta)-\int fd\rho}^{2}\leqslant K e^{-ct}\abs{V(\eta)}^{2} \quad  \text{ for all } t\geqslant 0
\end{equation} for constants $K,c>0$ (not depending on $f$) 
and for all $f:\mathbb{T}^{2}\times\R^{2K}\to\R$ such that $\abs{f(y,\eta)}\leqslant V(\eta)$, $(y,\eta)\in\mathbb{T}^{2}\times\R^{2K}$. If we integrate the pointwise inequality \eqref{eq:p-sp} over $(y,\eta)$ with respect to $\rho$, we obtain the $L^{2}(\rho)$-spectral-gap-type estimates for all such $f$, assuming $V\in L^{2}(\rho)$,
\begin{equation}\label{eq:sp-l-2}
 \norm{P_{t}^{(Y,\eta)}f-\int fd\rho}_{L^{2}(\rho)}^{2}\leqslant K e^{-ct}\norm{V}_{L^{2}(\rho)}^{2} \quad  \text{ for all } t\geqslant 0.
\end{equation}
We will in particular apply the spectral gap estimates  to $f=F$, which satisfies \eqref{eq:F}.\\
Let us, similarly as in the previous section, define the $H^{1}$ space with respect to the generator $\mathcal{G}$ (notation: $\mathcal{G}^{S}:=\frac{1}{2}(\mathcal{G}+\mathcal{G}^{\star})$, $\mathcal{G}^{\star}$ being the $L^{2}(\rho)$-adjoint),
\begin{align*}
H^{1}(\rho):=\{u\in dom(\mathcal{G})\mid \langle (-\mathcal{G})u,u\rangle_{\rho}=\langle (-\mathcal{G}^{S})u,u\rangle_{\rho}<\infty\}.
\end{align*}
The scalar product in $H^{1}(\rho)$ is given by $\langle f, g\rangle_{H^{1}(\rho)}=\langle (-\mathcal{G}^{S})f,g\rangle_{H^{1}(\rho)}$.\\
Then, as a consequence of the spectral gap estimates, we can solve the Poisson equation
\begin{equation}\label{eq:G-P-eq}
(-\mathcal{G})u=g 
\end{equation} explicitly with right-hand side $g$ that has mean zero under $\rho$, $\langle g\rangle_{\rho}=0$, and satisfies $\abs{g}\leqslant V$ with $V\in L^{2}(\rho)$. The unique solution $u \in H^{1}(\rho)$ is given by $u=\int_{0}^{\infty} P_{t}^{(Y,\eta)}gdt\in L^{2}(\rho) $.
In fact, for our tightness arguments, we will need a stronger integrability condition on the solution $u$ (and $\nabla_{\eta} u,\nabla_{y}u$), that is given by the following proposition.
\begin{proposition}\label{prop:w1p}
Let $p\geqslant 2$ and let $g\in C^{\infty}(\mathbb{T}^{2}\times\R^{2K},\R^{2})$ with 
\begin{align}\label{eq:growth-cond}
\abs{g(y,\eta)}\leqslant V(\eta)
\end{align} (i.p. $g\in L^{p}(\rho)$) and with $\langle g\rangle_{\rho}=\int_{\mathbb{T}^{2}\times\R^{2K}}g(y,\eta)\rho(d(y,\eta))=0$. Then the Poisson equation
\begin{align*}
(-\mathcal{G})u=g
\end{align*} has a unique strong solution $u\in C^{\infty}(\mathbb{T}^{2}\times\R^{2K},\R^{2})$ with the property that $\langle u\rangle_{\rho}=0$. Moreover, there exists a constant $C>0$, such that the solution satisfies $\abs{u(y,\eta)}\leqslant C V(\eta)$ and
\begin{align}
\abs{\nabla_{(y,\eta)}u(y,\eta)}\leqslant\abs{\nabla_{y} u(y,\eta)}+\abs{\nabla_{\eta} u(y,\eta)}\leqslant  2C V(\eta).
\end{align}
In particular, it follows that $u\in W^{1,p}(\rho)=\{u\in L^{p}(\rho)\mid \nabla_{(y,\eta)}u\in L^{p}(\rho)\}$.
\end{proposition}
\begin{proof}
The solution $u=\int_{0}^{\infty}P^{(Y,\eta)}_{t}gdt$ is smooth, as $g$ is assumed to be smooth, cf. also \cite[Proposition A.3.1]{duncan}. It satisfies an analogue growth bound as $g$ by the pointwise spectral gap estimates \eqref{eq:p-sp} with constant $C=K\int_{0}^{\infty}e^{-ct}dt\in (0,\infty)$, in particular $u\in L^{p}(\rho)$ for any $p\geqslant 1$. For the bound on the derivative, we proceed as in part (e) of the proof of \cite[Theorem 1]{pv} applying Sobolev embedding and the estimate $(9.40)$ from \cite{gt} and the bound on $g$ and $u$, such that for $p>d+2K$ (notation: $B_{x,R}=\{z\in\R^{d}\times\R^{2K}\mid\abs{z-x}\leqslant R\}$)
\begin{align*}
\abs{\nabla_{(y,\eta)}u(y,\eta)}\leqslant C (\norm{u}_{L^{p}(B_{(y,\eta),2})}+\norm{\mathcal{G}u}_{L^{p}(B_{(y,\eta),2})})\leqslant 2C V(\eta).
\end{align*}
Notice also that, compared to \cite{pv} in our situation, we have compactness in the $y$ variable and the bound on $g$ and $u,\nabla u$ is uniform in $y\in\mathbb{T}^{d}$.
\end{proof}

\noindent In what follows, we will always assume the system (\ref{eq:xye})  starts in stationarity , i.e. $(Y^{\epsilon}_{0},\eta^{\epsilon}_{0})\sim\rho$. In \cite[Theorem 7]{dep} the authors proved the homogenization result for the process $X^{\epsilon}$ itself (cf. also \cite[chapter 5]{duncan}), namely 
\begin{align*}
X^{\epsilon} \Rightarrow \sqrt{2D}Z,
\end{align*} 
where the convergence is in distribution in $C([0,T],\R^{2})$ when $\epsilon\to 0$ with $Z$ being a standard two-dimensional Brownian motion and 
\[
D= \int (I + \nabla_y \chi(y,\eta))^T\Sigma(y,\eta)(I + \nabla_y \chi(y,\eta)) \rho(dy, d\eta) + \int \nabla_\eta \chi(y,\eta)^T \Gamma \Pi \nabla_\eta \chi(y,\eta) \rho(dy, d\eta).
\]
Here $I\in\R^{2\times 2}$ is the identity matrix and $\chi$ is the solution of the Poisson equation $(-\mathcal{G})\chi=F$. 
To solve the Poisson equation with right-hand-side $F$, we furthermore need that $F$ is centered with respect to $\rho$, which is stated in the following lemma, and also in  \cite[Proposition 5.3.4]{duncan}.\\
In order to obtain the homogenization result for the rough path lift of the process $X^{\epsilon}$, we will use martingale methods (cf. in the book \cite[Ch. 2]{klo}) applied to the stationary, ergodic Markov process $(Y^{\epsilon},\eta^{\epsilon})$ started in $\rho$. In addition we will exploit the decomposition of the additive functional  in terms of Dynkin's martingale and the boundary term involving the solution of the Poisson equation \eqref{eq:G-P-eq}.

\begin{lemma}\label{cor:meanzero}\cite[Proposition 5.3.4]{duncan}\\
For $F$ from \eqref{eq:defF} and the invarinat probability measure $\rho$ for $\mathcal{G}$, the following centering condition holds true
\begin{align}
\int_{\mathbb{T}^{2}\times\R^{2K}}F(y,\eta)\rho(y,\eta)d(y,\eta)=0.
\end{align}
\end{lemma}

\noindent We prove in the following lemma, that the density $\rho$ that solves (\ref{eq:rho}) is given by $\rho(y,\eta)=g_{\eta}(y)f(\eta)$, where $f$ is the density of the Normal distribution invariant for $\eta$ and $g$ solves the equation \eqref{eq:g_eta} below.
\begin{lemma}\label{lem:rhostr}
Let $\rho$ be the probability measure with the density  denoted also by $\rho$, such that it solves $\mathcal{G}^{\ast}\rho=0$. Moreover, let  $g_{\eta}(y)$ be the unique solution, satisfying $\int_{\mathbb{T}^{2}}g_{\eta}(y)dy=1$ and $g_{\eta}(y)\geqslant 0$, to the equation
\begin{align}\label{eq:g_eta}
(\mathcal{L}_{0}^{\ast}+\mathcal{L}_{\eta})g_{\eta}(y)=0
\end{align} 
for the adjoint operator $\mathcal{L}_{0}^{\ast}=\mathcal{L}_{0}^{\ast}(\eta)$ of $\mathcal{L}_{0}=\mathcal{L}_{0}(\eta)$ with respect to $L^{2}(dy)$. Then  the density $\rho$ fulfills the disintegration formula
\begin{align*}
\rho(y,\eta)=g_{\eta}(y)f(\eta),
\end{align*}
where
\begin{align*}
f(\eta)=\frac{1}{2\pi\sqrt{\abs{\Pi}}}\exp\left(-\frac{1}{2}\eta^{T}\Pi^{-1}\eta\right).
\end{align*} 
In particular, the marginal distribution of $\rho$ in the $\eta$-variable is the normal distribution $N(0,\Pi)$.
\end{lemma}
\begin{proof}
First we show that the form of the density $\rho$ follows from the disintegration theorem from measure theory
(see for example \cite[chapter 3, 70 and 71]{Dellacherie-Meyer}) and the invariance of $\rho$. Let $\pi:\R^{2K}\times\mathbb{T}^{2}\to\R^{2K}, (\eta,y) \mapsto \eta$ be the projection and $\nu:=\rho\circ\pi^{-1}$ the push-forward under $\rho$. Then the disintegration theorem implies that there exists a family of measures $(\mu_\eta)_{\eta\in\R^{2K}}$  on  $\mathbb{T}^{2}$, such that:

$\bullet$   $\eta\mapsto\mu_{\eta}(A)$ is Borel measurable for each Borel measurable set $A\in \mathcal{B}(\mathbb{T}^{2})$

$\bullet$ for every Borel measurable function $h:\mathbb{T}^2 \times \R^{2K} \to \R$, 
\begin{equation}
\int_{\mathbb{T}^2 \times \R^{2K}} h(y,\eta)\rho(d(y,\eta))=\int_{\R^{2K}}\int_{\mathbb{T}^2} h(y,\eta)\mu_{\eta}(dy)\nu(d\eta).
\end{equation}
Since by assumption $\rho$ has a density, that we also denote by $\rho$, it follows that $\nu$ has a density given by 
\begin{equation*}
\eta\mapsto\int\rho(y,\eta)dy=:f(\eta).
\end{equation*}
Consequently, also $\mu_{\eta}$ has a density, namely the conditional density 
\begin{equation*}
y\mapsto \mathbf{1}_{\{f>0\}}\rho(y,\eta)/f(\eta)=:g_{\eta}(y).
\end{equation*}
In order to prove that the marginal distribution $\nu$ under $\rho$ is the normal distribution $N(0,\Pi)$, consider $h\in C_{b}(\mathbb{T}^{2}\times\R^{K},\R)$ with $h(y,\eta)=h(\eta)$ not depending on $y$. Then if $(Y_{0},\eta_{0})\sim\rho$, we have for any $t\geqslant 0$:
\begin{equation*}
\E[h(\eta_{t})]=\E[h(Y_{t},\eta_{t})]=\int h(y,\eta)d\rho=\int h(\eta)f(\eta)d\eta.
\end{equation*}
Hence $f$ is given by the density of the $N(0,\Pi)$ distribution, as this is the unique invariant distribution for $(\eta_{t})_t$. 

It is left to derive the equation (\ref{eq:g_eta}) for the density $g_{\eta}(y)$. For that we use the invariance of $\rho$ and write
\begin{align*}
0=\mathcal{G}^{\ast}\rho=(\mathcal{L}_{0}^{\ast}+\mathcal{L}_{\eta}^{\ast})(g_{\eta}(y)f(\eta))=f(\eta)\mathcal{L}_{0}^{\ast}g_{\eta}(y)+\mathcal{L}_{\eta}^{\ast}(g_{\eta}(y)f(\eta))=f(\eta)\paren[\big]{\mathcal{L}_{0}^{\ast}g_{\eta}(y)+\mathcal{L}_{\eta}g_{\eta}(y)},
\end{align*} 
where we used that for any $h\in C^{2}(\R^{2K},\R)$, 
\begin{align*}
\mathcal{L}_{\eta}^{\ast}(h(\eta)f(\eta))&=  \nabla_{\eta}\cdot (h(\eta)\Gamma\eta f(\eta))+\Gamma\Pi :\nabla_{\eta}\nabla_{\eta}(h(\eta)f(\eta))\\&=h(\eta)[ \nabla_{\eta}\cdot (\Gamma\eta f(\eta))+\Gamma\Pi :\nabla_{\eta}\nabla_{\eta} f(\eta)]\\&\quad +f(\eta)[\Gamma\eta\cdot\nabla_{\eta}h(\eta)+\Gamma\Pi : \nabla_{\eta}\nabla_{\eta} h(\eta)]\\&\quad + 2\nabla_{\eta}h(\eta)\cdot\Gamma\Pi\nabla_{\eta}f(\eta)
\\&=h(\eta)\mathcal{L}_{\eta}^{\ast} f(\eta)+f(\eta)\mathcal{L}_{\eta}h(\eta)\\&\quad +2(\nabla_{\eta}h(\eta)\cdot\Gamma\Pi\nabla_{\eta}f(\eta)+f(\eta)\Gamma \eta\cdot\nabla_{\eta}h(\eta))\\&=h(\eta)\mathcal{L}_{\eta}^{\ast}f(\eta)+f(\eta)\mathcal{L}_{\eta}h(\eta)\\&=f(\eta)\mathcal{L}_{\eta}h(\eta),
\end{align*} where we added and subtracted the term $f(\eta)\Gamma \eta\cdot\nabla_{\eta}h(\eta)$ and then used that $\nabla_{\eta}f(\eta)=-f(\eta)\Pi^{-1}\eta $ and $\mathcal{L}_{\eta}^{\ast}f=0$. As $f>0$, the previous implies the equation (\ref{eq:g_eta}) for $g_{\eta}(y)$.

The uniqueness of the solution $g_{\eta}(y)$ in the class of probability densities in the $y$-variable  follows from the uniqueness of the density $\rho$ solving $\mathcal{G}^{\ast}\rho=0$. Indeed, let $g^{1},g^{2}\in C^{2}(\R^{2K}\times\mathbb{T}^{2},\R)$ be positive such that $\int_{\mathbb{T}^{2}}g^{i}(\eta,y)dy=1$ and they solve 
\begin{equation*}
(\mathcal{L}_{0}^{\ast}+\mathcal{L}_{\eta})g^{i}(\eta,y)=0 \quad  \text{for } i=1,2.
\end{equation*}
Then setting $\rho^{i}(\eta,y):=g^{i}(\eta,y)f(\eta)$ for $i=1,2$ we obtain  probability densities of a probability measure on $\R^{2K}\times\mathbb{T}^{2}$ solving $\mathcal{G}^{\ast}\rho^{i}=0$ for $i=1,2$.  As a consequence $\rho^{1}=\rho^{2}=\rho$, which implies  $g^{1}=g^{2}$.
\end{proof}

\begin{subsection}{Determining the limit rough path}\label{4.1}
In this subsection, we prove the convergence of the Itô integrals 
\begin{align*}
\int_{s}^{t}(X^{\epsilon}_r-X^{\epsilon}_s)\otimes dX^{\epsilon}_r=\paren[\bigg]{\int_{s}^{t}(X^{\epsilon,i}_r-X^{\epsilon,i}_s)dX^{\epsilon,j}_r}_{i,j=1,2}
\end{align*} and determine the limit. In order to obtain the limit,  we will use a decomposition of $X^\epsilon(t)$ via the solution $\chi$ of the Poisson equation $\mathcal{G}\chi=-F$, which exists by \cref{prop:w1p}. 
Rewriting the drift term with Itô-formula for $\chi(Y^{\epsilon}_{t},\eta^{\epsilon}_{t})$,  we obtain:
\begin{align*}
\frac{1}{\epsilon}\int_{0}^{t}F(Y^{\epsilon}_s,\eta^{\epsilon}_s)ds=-(\epsilon(\chi(Y^{\epsilon}_{t},\eta^{\epsilon}_{t})-\chi(Y^\epsilon_0,\eta^\epsilon_0))- \tilde{M}^{\epsilon}_{t})
\end{align*}
where 
\begin{equation*} 
\tilde{M}^{\epsilon,i}_{t}:=\tilde{M}^{\epsilon,i}_{1}(t)+\tilde{M}^{\epsilon,i}_{2}(t):=\int_{0}^{t}\nabla_{y}\chi^{i} (Y^{\epsilon}_{s},\eta^{\epsilon}_{s})\cdot\sqrt{2\Sigma}(Y^{\epsilon}_{s},\eta^{\epsilon}_{s})dB_{s}+\int_{0}^{t}\nabla_{\eta}\chi^{i}(Y^{\epsilon}_{s},\eta^{\epsilon}_{s})\cdot\sqrt{2\Gamma\Pi} dW_{s}
\end{equation*} for $i=1,2$.
As a consequence, we have
\begin{equation}\label{XIto}
X^\epsilon_t = X^\epsilon_0 + \epsilon\left(\chi \left(Y^\epsilon_0, \eta^\epsilon_0\right) - \chi(Y^\epsilon_t, \eta^\epsilon_t) \right) + M_1^\epsilon (t) + M_2^\epsilon(t)
\end{equation}
where martingale terms $M^{\epsilon}:=M_1^{\epsilon}+M_{2}^{\epsilon}$ are given by
\begin{align}
M_1^{\epsilon,i}(t) &:= \int_0^t (\nabla_y \chi^{i}(Y^{\epsilon}_s,\eta^{\epsilon}_s) + e_{i})\cdot\sqrt{2 \Sigma} (Y^{\epsilon}_s,\eta^{\epsilon}_s) dB_s  \label{M1}\\
M_2^{\epsilon,i}(t) &:= \int_0^t \nabla_\eta \chi^{i}(Y^{\epsilon}_s,\eta^{\epsilon}_s)\cdot\sqrt{2 \Gamma \Pi}  dW_{s} \label{M2}
\end{align} for $i=1,2$.
Using the dynamics of $X^{\epsilon}$, we decompose the iterated integrals for $i,j\in\{1,2\}$,
\begin{align}
\int_0^t (X_s^{\epsilon,i} - X_0^{\epsilon,i}) dX_s^{\epsilon,j} 
&= \int_0^t(X_s^{\epsilon,i} - X_0^{\epsilon,i})  \sum_{l=1}^{2}\sqrt{2 \Sigma}(j,l) (Y^\epsilon_s, \eta^\epsilon_s ) dB^{l}_s \label{eq:mart}\\&\quad + \int_0^t  (X_s^{\epsilon,i} - X_0^{\epsilon,i}) \frac{1}{\epsilon}F^{j}(Y^\epsilon_s, \eta^\epsilon_s ) ds. \label{eq:drift}
\end{align}
The next step is to rewrite the terms \eqref{eq:mart} and \eqref{eq:drift} collecting the vanishing and non-vanishing terms.

\noindent First we consider the term \eqref{eq:mart} and plug the decomposition (\ref{XIto}) of $X^{\epsilon}$. We obtain
\begin{align}\label{eq:sum1}
\MoveEqLeft
\int_0^t(X_s^{\epsilon,i} - X_0^{\epsilon,i})  \sum_{l=1}^{2}\sqrt{2 \Sigma}(j,l) (Y^\epsilon_s, \eta^\epsilon_s ) dB^{l}_s\nonumber
\\&=\epsilon\int_{0}^{t}(\chi^{i}(Y^\epsilon_0, \eta^\epsilon_0)-\chi^{i}(Y^{\epsilon}_s,\eta^{\epsilon}_s))\sum_{l=1}^{2}\sqrt{2 \Sigma}(j,l) (Y^\epsilon_s, \eta^\epsilon_s ) dB^{l}_s\nonumber
\\&\qquad+\int_{0}^{t}M^{\epsilon,i}_{s}\sum_{l=1}^{2}\sqrt{2 \Sigma_{j,l} (Y^\epsilon_s, \eta^\epsilon_s )}dB^{l}_{s}\nonumber
\\&=\epsilon\int_{0}^{t}(\chi^{i}(Y^\epsilon_0, \eta^\epsilon_0)-\chi^{i}(Y^{\epsilon}_s,\eta^{\epsilon}_s))\sum_{l=1}^{2}\sqrt{2 \Sigma}(j,l) (Y^\epsilon_s, \eta^\epsilon_s ) dB^{l}_s\nonumber
\\&\qquad+ M_{t}^{\epsilon,i}\paren[\bigg]{\int_{0}^{t}\sum_{l=1}^{2}\sqrt{2 \Sigma}(j,l) (Y^\epsilon_r, \eta^\epsilon_r )dB^{l}_{r}}\nonumber
\\&\quad-\int_{0}^{t}\paren[\bigg]{\int_{0}^{s}\sum_{l=1}^{2}\sqrt{2 \Sigma}(j,l) (Y^\epsilon_r, \eta^\epsilon_r )dB^{l}_{r}}dM^{\epsilon,i}_{s}-\angles[\bigg]{ \int_{0}^{\cdot}\sum_{l=1}^{2}\sqrt{2 \Sigma}(j,l) (Y^\epsilon_r, \eta^\epsilon_r )dB^{l}_{r},M^{\epsilon,i}}_{t}.
\end{align} 
By stationarity of $(Y^{\epsilon},\eta^{\epsilon})$ and the boundedness (\ref{eq:Sigma}) of $\Sigma$, 
we notice that the first summand in the decomposition \eqref{eq:sum1} will converge in $L^{2}(\p)$ to zero.  Moreover, for the quadratic variation term, we can argue with the ergodic theorem for $(Y,\eta)$, \cite[Theorem 3.3.1]{DaPrato1996}, obtaining the convergence in probability
\begin{align}\label{eq:qv-conv}
\MoveEqLeft
\p\bigg(\abs[\bigg]{\angles[\bigg]{ \int_{0}^{\cdot}\sum_{l=1}^{2}\sqrt{2 \Sigma}(j,l) (Y^\epsilon_r, \eta^\epsilon_r )dB^{l}_{r},M^{\epsilon,i}}_{t}-t\int e_{j}\cdot 2\Sigma (e_{i}+\nabla_{y}\chi^{i})d\rho}>\delta\bigg)\nonumber\\
&=\p\bigg(\abs[\bigg]{\int_{0}^{t}e_{j}\cdot 2\Sigma (e_{i}+\nabla_{y}\chi^{i})(Y^{\epsilon}_{s},\eta^{\epsilon}_{s})ds-t\int e_{j}\cdot 2\Sigma (e_{i}+\nabla_{y}\chi^{i})d\rho}>\delta\bigg)\nonumber\\
&=\tilde{\p}\bigg(\abs[\bigg]{\epsilon^{2}\int_{0}^{\epsilon^{-2}t}e_{j}\cdot 2\Sigma (e_{i}+\nabla_{y}\chi^{i})(Y_{s},\eta_{s})ds- t\int e_{j}\cdot 2\Sigma (e_{i}+\nabla_{y}\chi^{i})d\rho}>\delta\bigg)\to 0,
\end{align} 
when $\epsilon\to 0$, for any $\delta>0$, using that $(Y^{\epsilon}_t,\eta^{\epsilon}_t)_{t\geqslant 0}\stackrel{d}{=}(Y_{\epsilon^{-2}t},\eta_{\epsilon^{-2}t})_{t\geqslant 0}$, where $(Y_t,\eta_t)_{t\geqslant 0}$ is the Markov process with generator $\mathcal{G}$ (with respect to some base probability measure $\tilde{\p}$). 
To deduce the convergence of the remaining two martingale terms in \eqref{eq:sum1}, we will add them up with the decomposition of the term \eqref{eq:drift} below.

\noindent 
We decompose the term \eqref{eq:drift} in the following way
\begin{align}
\int_0^t  (X_s^{\epsilon,i} - X_0^{\epsilon,i}) \frac{1}{\epsilon}F^{j}(Y^\epsilon_s, \eta^\epsilon_s ) ds&=\int_0^t  (\chi^{i}(Y^\epsilon_0, \eta^\epsilon_0)-\chi^{i}(Y^{\epsilon}_s,\eta^{\epsilon}_s)) F^{j}(Y^\epsilon_s, \eta^\epsilon_s ) ds\label{driftTerm1}\\
&\qquad+\int_0^t  M_{s}^{\epsilon,i} \frac{1}{\epsilon}F^{j}(Y^\epsilon_s, \eta^\epsilon_s ) dt\label{driftTerm2}.
\end{align}
For the first term in (\ref{driftTerm1}) we again apply the ergodic theorem for $(Y,\eta)$ yielding the convergence in probability, analogously as above,
\begin{align*}
\MoveEqLeft
\int_{0}^{t}(\chi^{i}(Y^\epsilon_0, \eta^\epsilon_0)-\chi^{i}(Y^{\epsilon}_s,\eta^{\epsilon}_s)) F^{j}(Y^{\epsilon}_s, \eta^{\epsilon}_s ) ds\\
&\stackrel{d}{=}\chi^{i}(Y^\epsilon_0, \eta^\epsilon_0)\epsilon^{2}\int_{0}^{\epsilon^{-2}t}F^{j}(Y_{r},\eta_{r})dr-\epsilon^{2}\int_{0}^{\epsilon^{-2}t}\chi^{i}F^{j}(Y_{r},\eta_{r})dr\\
&\rightarrow t(\chi^{i}(Y^\epsilon_0,\eta^\epsilon_0)E_{\rho}[F^{j}]-E_{\rho}[\chi^{i} F^{j}])=tE_{\rho}[\chi^{i}(-F)^{j}]=:t a_{F}(i,j),
\end{align*} 
where we used that $F$ has mean zero under $\rho$ by \cref{cor:meanzero} and we introduced the notation $E_{\rho}[f]:=\int f(y,\eta)\rho(y,\eta)d(y,\eta)$.

\noindent For the second term \eqref{driftTerm2}, we apply the integration by parts formula to further rewrite
\begin{align}\label{eq:sum2}
\int_{0}^{t}M^{\epsilon, i}_{s}\epsilon^{-1}F^{j}\left(Y_{s}^{\epsilon},\eta_{s}^{\epsilon}\right)ds=M^{\epsilon,i}_{t}\int_{0}^{t}\epsilon^{-1}F^{j}(Y_{s}^{\epsilon},\eta_{s}^{\epsilon})ds-\int_{0}^{t}\left(\int_{0}^{s}\epsilon^{-1}F^{j}(Y_{r}^{\epsilon},\eta_{r}^{\epsilon})dr \right)dM^{\epsilon,i}_{s}.
\end{align}

\noindent Let  $a_F^{\epsilon}$  be defined as
\begin{align*}
a_{F}^{\epsilon}(i,j)&:=\epsilon\int_{0}^{t}(\chi^{i}(Y^\epsilon_0, \eta^\epsilon_0)-\chi^{i}(Y^{\epsilon}_s,\eta^{\epsilon}_s))d\paren[\bigg]{\int_{0}^{\cdot}\sqrt{2 \Sigma} (Y^\epsilon, \eta^\epsilon  ) dB}^{j}_s \\&\quad+  \int_0^t  (\chi^{i}(Y^\epsilon_0, \eta^\epsilon_0)-\chi^{i}(Y^{\epsilon}_s,\eta^{\epsilon}_s) )F^{j}(Y^\epsilon_s, \eta^\epsilon_s ) ds.
\end{align*}

\noindent Then, using the definition of $a_{F}^\epsilon$ and summing up the two remaining terms in \eqref{eq:sum1} and the terms in \eqref{eq:sum2}, we get
\begin{align}\label{eq:itint}
\int_{0}^{t}(X^{\epsilon,i}_{s}-X^{\epsilon,i}_{0})dX_{s}^{\epsilon,j}
&{=a_{F}^{\epsilon}(i,j)+M_{t}^{\epsilon,i}(X_{t}^{\epsilon,j}-X_{0}^{\epsilon,j})-\int_{0}^{t}(X_{s}^{\epsilon,j}-X_{0}^{\epsilon,j})dM_{s}^{\epsilon,i}}\nonumber\\
&\quad {-\angles[\bigg]{ M^{\epsilon,i},\paren[\bigg]{\int_{0}^{\cdot}\sqrt{2\Sigma }(Y^{\epsilon}_{r},\eta^{\epsilon}_{r})dB_{r}}^{j}}_{t}}.
\end{align}
To obtain the limit in distribution, we utilize the convergence of $a_{F}^{\epsilon}$ in probability proven above (using the convergence of the term \eqref{driftTerm1} and the term vanishing in $L^{2}(\p)$), the convergence of the quadratic variation term in \eqref{eq:qv-conv} and \cref{prop:UCV} for the remaining terms. Here Slutzky's lemma ensures that the sum of a random variable converging in distribution and a random variable converging in probability, converges in distribution to the sum of the limits. To apply \cref{prop:UCV}, we check that the UCV condition for $(M^{\epsilon,i})$ is satisfied and that $(M^{\epsilon,i},X^{\epsilon,j})\Rightarrow (X^{i}-X^{i}_{0},X^{j})$ jointly in distribution. Here, the  joint convergence is due the decomposition \eqref{XIto} and the convergence for the process by \cite[Theorem 7]{dep}. To show the UCV condition, we utilize the stationarity of $(Y,\eta)$ and $(Y^{\epsilon},\eta^{\epsilon})\stackrel{d}{=}(Y_{\epsilon^{-2}\cdot},\eta_{\epsilon^{-2}\cdot})$, such that
\begin{align*}
\E[\langle \tilde{M}^{\epsilon,i}\rangle_{t}]
=t\int [(\nabla_{y}\chi^{i})^{T}2\Sigma\nabla_{y}\chi^{i}+(\nabla_{\eta}\chi^{i})^{T}2\Gamma\Pi\nabla_{\eta}\chi^{i}]d\rho<\infty.
\end{align*} 
Here the right-hand-side is finite due to \cref{prop:w1p} and boundedness of $\Sigma$, \eqref{eq:Sigma}, and it does not depend on $\epsilon$, such that, utilizing again boundedness of $\Sigma$, the UCV condition for $(M^{\epsilon,i})_{\epsilon}$ is satisfied.\\
All together, we thus obtain the distributional convergence,
\begin{align*}
\int_{0}^{t}(X^{\epsilon,i}_{s}-X^{\epsilon,i}_{0})dX_{s}^{\epsilon,j}
&{=a_{F}^{\epsilon}(i,j)+M_{t}^{\epsilon,i}(X_{t}^{\epsilon,j}-X_{0}^{\epsilon,j})-\int_{0}^{t}(X_{s}^{\epsilon,j}-X_{0}^{\epsilon,j})dM_{s}^{\epsilon,i}}\nonumber\\
&\quad {-\angles[\bigg]{ M^{\epsilon,i},\paren[\bigg]{\int_{0}^{\cdot}\sqrt{2\Sigma }(Y^{\epsilon}_{r},\eta^{\epsilon}_{r})dB_{r}}^{j}}_{t}}\\&\Rightarrow ta_{F}(i,j)+(X^{i}_{t}-X^{i}_{0})(X^{j}_{t}-X_{0}^{j})-\int_{0}^{t}(X^{j}_{s}-X_{0}^{j})dX_{s}^{i}\nonumber\\&\quad-t\int e_{j}\cdot 2\Sigma (e_{i}+\nabla_{y}\chi^{i})d\rho\nonumber\\
&=t a_{F}(i,j)+\langle X^{i},X^{j}\rangle_{t}-t\int (e_{i}+\nabla_{y}\chi^{i})\cdot 2\Sigma e_{j}d\rho+\int_{0}^{t}(X^{i}_{s}-X_{0}^{i})dX^{j}_{s}.\nonumber
\end{align*} 
The arguments can also be generalized to a base-point $s>0$ in the same manner as for $s=0$ above, such that from above we obtain the weak limit of the iterated integrals $\mathbb{X}_{s,t}^{\epsilon}(i,j)$, which decomposes in the iterated integrals $\mathbb{X}_{s,t}(i,j)$ of the Brownian motion $X=\sqrt{2D}Z$ plus an area correction term. Furthermore, the joint distributional convergence of $((\mathbb{X}^{\epsilon}_{s,t}(i,j))_{i,j=1,2})_{\epsilon}$ follows from the decomposition \eqref{eq:itint} and joint distributional convergence of $((M^{\epsilon,i},X^{\epsilon,j})_{i,j=1,2})_{\epsilon}$, which relies on joint convergence of $((X^{\epsilon,i})_{i=1,2})_{\epsilon}$ by \cite[Theorem 7]{dep}.\\
We summarize our findings in the following proposition.
\begin{proposition}\label{prop:limit1}
Let $(X^{\epsilon},Y^{\epsilon},\eta^{\epsilon})$ be the solution of the system \eqref{eq:xye} for $(Y^{\epsilon}_{0},\eta^{\epsilon}_{0})\sim\rho$. Then for all $s,t\in\Delta_{T}$,  the iterated Itô-integrals $(\mathbb{X}^{\epsilon}_{s,t})$   convergence weakly in $\R^{2\times 2}$, as $\epsilon\to 0$ 
\begin{align}
\mathbb{X}^{\epsilon}_{s,t}(i,j):=\int_{s}^{t}X^{\epsilon,i}_{s,r}dX^{\epsilon,j}_{r}&\Rightarrow \mathbb{X}_{s,t}(i,j)+(t-s)\paren[\bigg]{\langle X^{i},X^{j}\rangle_{1}+\langle \chi^{i},(\mathcal{G}\chi)^{j}\rangle_{\rho}\\&\quad-\int (e_{i}+\nabla_{y}\chi^{i})\cdot 2\Sigma e_{j}d\rho}, \quad  i,j\in\{1,2\}
\end{align} 
where
\begin{align*}
X=\sqrt{2D}Z,\quad\mathcal{G}\chi=-F,\quad e_{1}=(1,0),\quad e_{2}=(0,1)
\end{align*} 
for a standard two-dimensional Brownian motion $Z$, $\mathbb{X}_{s,t}(i,j):=\int_{s}^{t}X_{s,r}^{i}dX^{j}_{r}$ and
\begin{align}\label{effectiveDiff}
D=\int (I + \nabla_y \chi)^{T}\Sigma(I + \nabla_y \chi) \rho(dy, d\eta) + \int (\nabla_\eta \chi)^{T} \Gamma \Pi \nabla_\eta \chi \rho(dy, d\eta).
\end{align}
\end{proposition}
\begin{corollary}\label{cor:limit1}
Let $(X^{\epsilon},Y^{\epsilon},\eta^{\epsilon})$ be as in \cref{prop:limit1}. Then for all $s,t\in\Delta_{T}$ also the iterated Stratonovich integrals $(\tilde{\mathbb{X}}_{s,t}^{\epsilon})$ converge weakly in $\R^{2\times 2}$ as $\epsilon\to 0$:
\begin{align}
\tilde{\mathbb{X}}^{\epsilon}_{s,t}(i,j):=\int_{s}^{t}X^{\epsilon,i}_{s,r}\circ dX^{\epsilon,j}_{r}\Rightarrow\tilde{\mathbb{X}}_{s,t}(i,j)+(t-s)\tilde{A}(i,j), \quad i,j\in\{1,2\}
\end{align} 
 where 
\begin{align*}
X=\sqrt{2D}Z,\quad \tilde{\mathbb{X}}_{s,t}(i,j):=\int_{s}^{t}X_{s,r}^{i}\circ dX^{j}_{r}
\end{align*} 
for a standard two-dimensional Brownian motion $Z$ and  $D$ is given by \eqref{effectiveDiff}.\\
Furthermore, the area correction is given by
\begin{align*}
\tilde{A}(i,j) = \langle\chi^{i},\mathcal{G}^{A}\chi^{j}\rangle_{\rho}+\int e_{i}\cdot\Sigma\nabla_{y}\chi^{j}d\rho-\int \nabla_{y}\chi^{i}\cdot \Sigma e_{j} d\rho.
\end{align*}
for $\mathcal{G}^{A}:=\frac{1}{2}(\mathcal{G}-\mathcal{G}^{\star})$, where $\mathcal{G}^{\star}$ denotes the $L^{2}(\rho)$-adjoint.
\end{corollary}
\begin{proof}
Recall the relation between the Itô and Stratonovich integral
\begin{align*}
\tilde{\mathbb{X}}^{\epsilon}_{0,t}(i,j)=\mathbb{X}^{\epsilon}_{0,t}(i,j)+\frac{1}{2}\langle X^{\epsilon,i},X^{\epsilon,j}\rangle_{t}.
\end{align*} 
The ergodic theorem for $(Y,\eta)$, \cite[Theorem 3.3.1]{DaPrato1996} together with $(Y^{\epsilon}_t,\eta^{\epsilon}_t)_{t\geqslant 0}\stackrel{d}{=}(Y_{\epsilon^{-2}t},\eta_{\epsilon^{-2}t})_{t\geqslant 0}$, implies the convergence in probability of the quadratic variation:
\begin{align*}
\frac{1}{2}\langle X^{\epsilon,i},X^{\epsilon,j}\rangle_{t}=\int_{0}^{t}e_{i}\cdot\Sigma e_{j} (Y^{\epsilon}_{s},\eta^{\epsilon}_{s})ds\to t\int e_{i}\cdot\Sigma e_{j} d\rho.
\end{align*} 
Thus we obtain, from \cref{prop:limit1} and \cref{lem:rhostr}, the following convergence in distribution
\begin{align*}
\tilde{\mathbb{X}}^{\epsilon}_{0,t}(i,j)=\mathbb{X}^{\epsilon}_{0,t}(i,j)+\frac{1}{2}\langle X^{\epsilon,i},X^{\epsilon,j}\rangle_{t}&\Rightarrow\mathbb{X}_{0,t}(i,j)+t\langle\chi^{i},(\mathcal{G}\chi)^{j}\rangle_{\rho}+\langle X^{i},X^{j}\rangle_{t}\\&\quad-t\int (\nabla_{y}\chi^{i}+e_{i})\cdot 2\Sigma e_{j}d\rho+t\int e_{i}\cdot \Sigma e_{j}d\rho\\& = \mathbb{X}_{0,t}(i,j)+ \frac{1}{2}\langle X^{i},X^{j}\rangle_{t} \\&\quad +t\paren[\bigg]{\langle\chi^{i},(\mathcal{G}\chi)^{j}\rangle_{\rho} + \frac{1}{2}\langle X^{i},X^{j}\rangle_{1} -\int \nabla_{y}\chi^{i}\cdot 2\Sigma e_{j}d\rho - \int e_{i}\cdot \Sigma e_{j}d\rho}\\&=\tilde{\mathbb{X}}_{0,t}(i,j)+t\tilde{A}(i,j).
\end{align*} 
The area correction can furthermore be written as
\begin{align*}
\tilde{A}(i,j) &= \langle\chi^{i},\mathcal{G}\chi^{j}\rangle_{\rho} + D(i,j) -\int \nabla_{y}\chi^{i}\cdot 2\Sigma e_{j}d\rho - \int e_{i}\cdot \Sigma e_{j}d\rho \\&=\langle\chi^{i},\mathcal{G}^{A}\chi^{j}\rangle_{\rho}+\int e_{i}\cdot \Sigma \nabla_{y}\chi^{j}d\rho-\int\nabla_{y}\chi^{i}\cdot\Sigma e_{j}d\rho,
\end{align*}
using that $\mathcal{G}=\mathcal{G}^{A}+\mathcal{G}^{S}$ for $\mathcal{G}^{S}:=\frac{1}{2}(\mathcal{G}+\mathcal{G}^{\star})$ and $\mathcal{G}^{A}:=\frac{1}{2}(\mathcal{G}-\mathcal{G}^{\star})$, where $\mathcal{G}^{\star}$ denotes the $L^{2}(\rho)$-adjoint, and that by \cite[section 2.4]{klo}, we have a correspondence between the quadratic variation of Dynkin's martingale $\tilde{M}^{\epsilon}$ and the $\mathcal{H}^{1}(\rho)$-norm of $\chi$ (utilizing also stationarity of $(Y,\eta)$ and $(Y^{\epsilon},\eta^{\epsilon})\stackrel{d}{=}(Y_{\epsilon^{-2}\cdot},\eta_{\epsilon^{-2}\cdot})$), such that
\begin{align*}
\langle\chi^{i},\mathcal{G}^{S}\chi^{j}\rangle_{\rho}=-\langle\chi^{i},\chi^{j}\rangle_{H^{1}(\rho)}=-\frac{1}{2}\E[\langle\tilde{M}^{\epsilon,i},\tilde{M}^{\epsilon,j}\rangle_{1}]=-\int[\nabla_{y}\chi^{i}\cdot\Sigma\nabla_{y}\chi^{j}+\nabla_{\eta}\chi^{i}\cdot\Gamma\Pi\nabla_{\eta}\chi^{j}]d\rho. 
\end{align*}
For a base-point $s>0$ we can do an analogue argument. 
\end{proof}

\begin{remark} 
We expect (without proof) that in fact $\tilde{A}(i,j)=\langle\chi^{i},\mathcal{G}^{A}\chi^{j}\rangle_{\rho}$ holds true and that $\mathcal{G}^{A}$ is a nontrivial operator, such that the Stratonovich area correction is truely non-vanishing. The problem to prove this is that, although we derived the form of the density $\rho$ in \cref{lem:rhostr}, the density $g_{\eta}(y)$ remains non-explicit.\\
Typically (cf.  \cite{radditive}) the area correction in the Stratonovich case can be expressed in terms of the asymmetric part of the generator of the underlying Markov process, which is a non-trivial operator if the Markov process is non-reversible. 
\end{remark}

\end{subsection}
\begin{subsection}{Tightness in $\gamma$-Hölder rough path topology}
For the convergence in distribution of the lift $(X^{\epsilon},\mathbb{X}^{\epsilon})$ to the respective lift of $X$ it is left to prove tightness in the rough path space utilizing \cref{lem:rp-kol}. We verify the necessary moment bounds in the next proposition.
\begin{proposition}\label{prop:tight1}
Let $(X^{\epsilon},Y^{\epsilon},\eta^{\epsilon})$ be as in \cref{prop:limit1}. Then the following moment bounds hold true for any $p\geqslant 2$
\begin{align}
\sup_{\epsilon}\E[\abs{X^{\epsilon,i}_t-X^{\epsilon,i}_s}^{p}]\lesssim\abs{t-s}^{p/2},\quad\forall s,t\in\Delta_{T}
\end{align} and
\begin{align}
\sup_{\epsilon}\E\bigg[\abs[\bigg]{\int_{s}^{t}X^{\epsilon, i}_{s,r} dX^{\epsilon,j}_{r}}^{p/2}\bigg]\lesssim\abs{t-s}^{p/2},\quad\forall s,t\in\Delta_{T}
\end{align} 
for $i,j\in\{1,2\}$. 
In particular, tightness of $(X^{\epsilon},\mathbb{X}^{\epsilon})$ in $C_{\gamma,T}$ for $\gamma<1/2$ follows.
\end{proposition}
\begin{proof}
Let $p\geqslant 2$. First, utilizing the growth condition  (\ref{eq:F}) on $F$, finiteness of moments of $\eta$, Burkholder-Davis-Gundy inequality and boundedness of $\Sigma$ in  (\ref{eq:Sigma}), we obtain
\begin{align}\label{eq:firsteq}
\E[\abs{X_{t}^{\epsilon}-X^{\epsilon}_{s}}^{p}]&{\leqslant\E\bigg[\abs[\bigg]{\epsilon^{-1}\int_{s}^{t}F(Y^{\epsilon}_{s},\eta^{\epsilon}_{s})ds}^{p}\bigg]+\E\bigg[\abs[\bigg]{\int_{s}^{t}\sqrt{2\Sigma}(Y^{\epsilon}_{s},\eta^{\epsilon}_{s})dB_{s}}^{p}\bigg]}\nonumber\\&\lesssim \epsilon^{-p}\abs{t-s}^{{p}}+\abs{t-s}^{p/2}.
\end{align}
Secondly,  using the representation \eqref{eq:itint}, we conclude
\begin{align*}
\E[\abs{X^{\epsilon,i}_{t}-X^{\epsilon,i}_{s}}^{p}]&\leqslant\epsilon^{p}\E[\abs{\chi^{i}(Y^{\epsilon}(t),\eta^{\epsilon}(t))-\chi^{i}(Y^{\epsilon}(s),\eta^{\epsilon}(s))}^{p}]+\E[\abs{M^{\epsilon,i}(t)-M^{\epsilon,i}(s)}^{p}],
\end{align*}
where the martingale is given by $M^{\epsilon}=M_1^{\epsilon}+M_{2}^{\epsilon}$ with
\begin{align*}
M_1^{\epsilon,i}(t) &:= \int_0^t (\nabla_y \chi^{i} + e_{i}) \cdot\sqrt{2 \Sigma}(Y^{\epsilon}(s),\eta^{\epsilon}(s)) dB_s \\
M_2^{\epsilon,i}(t) &:= \int_0^t  \nabla_\eta \chi^{i}(Y^{\epsilon}(s),\eta^{\epsilon}(s))\cdot \sqrt{2 \Gamma \Pi}dW_{s}. 
\end{align*}
Burkholder-Davis-Gundy (BDG) and Minkowski inequality, together with the stationarity of $(Y^{\epsilon},\eta^{\epsilon})$ yield
\begin{align*}
\E[\abs{M^{\epsilon}(t)-M^{\epsilon}(s)}^{p}]\lesssim\abs{t-s}^{p/2}E_{\rho}[\abs{(\nabla_{y}\chi+I)^{T}\Sigma(\nabla_{y}\chi+I)+(\nabla_{\eta}\chi)^{T}\Gamma\Pi\nabla_{\eta}\chi}^{p/2}].
\end{align*} 
Boundedness of $\Sigma$ stated in (\ref{eq:Sigma}) and \cref{prop:w1p}, i.e. $\nabla_{y}\chi+\nabla_{\eta}\chi\in L^{p}(\rho)$ for any $p\geqslant 2$, imply finiteness of the expectation on the right-hand-side. 

\noindent Moreover, according to \cref{prop:w1p}, $\chi$ satisfies a growth condition in $\eta$, which we use to estimate the boundary term (as well as stationarity of $\eta^{\epsilon}$ and that $\eta_{0}$ has all moments under $\rho$, as it is the normal distribution $N(0,\Pi)$), obtaining 
\begin{align}\label{eq:seceq}
\E[\abs{X^{\epsilon,i}_{t}-X^{\epsilon,i}_{s}}^{p}]&\lesssim\epsilon^{p}+\abs{t-s}^{p/2}.
\end{align}
By combining the estimates \eqref{eq:firsteq} and \eqref{eq:seceq}, using the first one for $\epsilon>\abs{t-s}^{1/2}$ and the second one for $\epsilon<\abs{t-s}^{1/2}$, we conclude
\begin{align*}
\sup_{\epsilon\in [0,1]}\E[\abs{X^{\epsilon,i}_{t}-X^{\epsilon,i}_{s}}^{p}]\lesssim\abs{t-s}^{p/2}.
\end{align*} 
The estimate for the iterated integrals is then immediate by the estimate on the moments of $X^{\epsilon,j}$ and the decomposition of the iterated integral in \eqref{eq:itint}, as well as the boundedness of the quadratic variation of the martingale $M^{\epsilon,i}$ in $L^{p/2}(\rho)$ for any $p\geqslant 2$.\\
The conclusion on tightness in $C_{\gamma,T}$ for $\gamma<1/2$ follows from \cref{lem:rp-kol}.
\end{proof}
\end{subsection}

\begin{subsection}{Rough homogenization limit in the $(\alpha,\beta)=(1,2)$-regime}
In this subsection, we state one of our main theorems, which is the corollary of \cref{prop:limit1} and \cref{prop:tight1}.
\begin{theorem}[Itô-lift]\label{thm:ito1}
Let $(X^{\epsilon},Y^{\epsilon},\eta^{\epsilon})$, $X$ and $\mathbb{X}^{\epsilon},\mathbb{X}$ be as in \cref{prop:limit1}. Then for any $\gamma<1/2$, $(X^{\epsilon},\mathbb{X}^{\epsilon}) $ weakly converges  in the $\gamma$-Hölder rough paths space, as $\epsilon \to 0$
\begin{align}
(X^{\epsilon},\mathbb{X}^{\epsilon}) \Rightarrow (X,(s,t)\mapsto\mathbb{X}_{s,t}+A(t-s)),
\end{align}
where $A=(A(i,j))_{i,j\in\{1,2\}}$ for
\begin{align}
A(i,j)=\langle\chi^{i},(\mathcal{G}\chi)^{j}\rangle_{\rho}+\langle X^{i},X^{j}\rangle_{1}-\int (e_{i}+\nabla_{y}\chi^{i})\cdot 2\Sigma e_{j}d\rho
\end{align} 
and $\chi$ being the solution of $\mathcal{G}\chi=-F$.
\end{theorem}
\begin{proof}
From \cref{prop:limit1}  the convergence of the one dimensional distributions of $(X^{\epsilon},\mathbb{X}^{\epsilon})_{\epsilon}$, that is of $(X^{\epsilon}_{t})$ and $(\mathbb{X}^{\epsilon}_{s,t})_{\epsilon}$ for any $0\leqslant s<t\leqslant T$, follows. By weak convergence of $(X^{\epsilon})$, it follows in particular convergence of the finite dimensional distributions. For the finite dimensional distributions of $\mathbb{X}^{\epsilon}$, we use the same argument as for the one dimensional distributions, noticing that the convergence in \eqref{eq:itint} of the part for which we applied the UCV condition, is also true as weak convergence of processes in $C(\Delta_{T},\R^{2\times 2})$ (due to the convergence of the processes from \cref{prop:UCV}) and the remaining terms converge in probability. Furthermore, \cref{prop:tight1} yields the tightness in the rough path space $C_{\gamma,T}$ for $\gamma<1/2$. Together, we obtain the weak convergence in $C_{\gamma,T}$ for $\gamma<1/2$: 
\begin{equation*}
(X^{\epsilon},\mathbb{X}^{\epsilon})\Rightarrow(X,(s,t)\mapsto\mathbb{X}_{s,t}+A(t-s)),
\end{equation*}
as claimed.
\end{proof}
\begin{corollary}[Stratonovich-lift]\label{cor:strato1}
Let $(X^{\epsilon},Y^{\epsilon},\eta^{\epsilon})$, $X$ and $\tilde{\mathbb{X}}^{\epsilon},\tilde{\mathbb{X}}$ be as in \cref{cor:limit1}. Then for any $\gamma<1/2$, we have the weak convergence in the rough path space $C_{\gamma,T}$
\begin{align}
(X^{\epsilon},\tilde{\mathbb{X}}^{\epsilon}) \Rightarrow (X,(s,t)\mapsto\tilde{\mathbb{X}}_{s,t}+\tilde{A}(t-s))
\end{align} when $\epsilon\to 0$, where
\begin{align*}
\tilde{A}(i,j)=\langle\chi^{i},\mathcal{G}^{A}\chi^{j}\rangle_{\rho}+\int e_{i}\cdot \Sigma \nabla_{y}\chi^{j}d\rho-\int\nabla_{y}\chi^{i}\cdot\Sigma e_{j}d\rho,
\quad\mathcal{G}\chi=-F.
\end{align*} 
\end{corollary}
\begin{proof}
The proof follows immediately from \cref{cor:limit1} and \cref{thm:ito1}.
\end{proof}

\end{subsection}
\end{section}

\begin{section}{Diffusion on a membrane with comparable spatial and temporal fluctuations}\label{sect:hom2}
In this section we consider the space-time scaling regime when $\alpha=1,\beta=1$, which means that we observe spatial and temporal fluctuations of a membrane of comparable size. Then the general system becomes
\begin{empheq}[left=\empheqlbrace]{align}\label{eq:system3}
 dX^\epsilon_t & = \frac{1}{\epsilon}F(Y^\epsilon_t, \eta^\epsilon_t ) dt + \sqrt{2 \Sigma (Y^\epsilon_t, \eta^\epsilon_t )} dB_t,  \nonumber\\
 dY^\epsilon_t & = \frac{1}{\epsilon^2}F(Y^\epsilon_t, \eta^\epsilon_t ) dt + \sqrt{\frac{2}{\epsilon^2} \Sigma (Y^\epsilon_t, \eta^\epsilon_t )} dB_t, \\
 d \eta^\epsilon_t & = - \frac{1}{\epsilon} \Gamma \eta^\epsilon_t dt + \sqrt{\frac{2}{\epsilon} \Gamma \Pi} dW_t\nonumber, 
\end{empheq}
where again $B$ and $W$ are independent Brownian motions and $Y^{\epsilon}:=\epsilon^{-1}X^{\epsilon}\text{ mod }\mathbb{T}^{2}$.\\
To determine the limit in this regime, the difficulty is that $Y$ and $\eta$ now fluctuate at different scales, which means  that, compared to the previous section, we no longer have the generator $\epsilon^{-2}\mathcal{G}$ for the joint Markov process $(Y^{\epsilon},\eta^{\epsilon})$, but the generator $\epsilon^{-2}\mathcal{L}_{0}+\epsilon^{-1}\mathcal{L}_{\eta}$ for $\mathcal{L}_{0}=\mathcal{L}_{0}(\eta)$ from \eqref{eq:L0} and $\mathcal{L}_{\eta}$ from \eqref{OUgen}. The idea is to first deduce a quenched result for each fixed environment $\eta$ and afterwards average over the invariant measure $\rho_{\eta}$ of $\eta$.

\noindent Let for $\eta\in\R^{2K}$, 
\begin{equation*}
\rho_{Y}(dy,\eta):=C^{-1}\sqrt{\abs{\Sigma^{-1}}(y,\eta)}dy
\end{equation*}
be a probability measure on $\mathbb{T}^{2}$, where $C:=\int_{\mathbb{T}^{2}} \sqrt{\abs{\Sigma^{-1}}(y,\eta)}dy$ is the normalizing constant. It is straightforward to  check that $\rho_{Y}(\cdot,\eta)$ is invariant for {$\mathcal{L}_{0}(\eta)$}, that is $\langle \mathcal{L}_{0}(\eta) f\rangle_{\rho_{Y}(\cdot,\eta)}=0$ for all $f\in dom(\mathcal{L}_{0})\subset L^{2}(\rho_{Y}(\cdot,\eta))$, and that $\int F(y,\eta)\rho_{Y}(dy,\eta)=0$ by the definition \eqref{eq:defF} of $F$. 

\noindent Let now, by \cite[Proposition A.2.2]{duncan}, for any $\eta\in\R^{K}$, $\chi(\cdot,\eta)\in C^{2}(\mathbb{T}^{2},\R^{2})$ be the unique solution of
\begin{equation}
 \mathcal{L}_{0}(\eta)\chi(\cdot,\eta)=-F(\cdot,\eta), \label{eq:eta}
 \end{equation} with $\int\chi(y,\eta)\rho_{Y}(dy,\eta)=0$. Existence of $\chi$ is based on the $L^{2}(\rho_{Y}(\cdot,\eta))$-spectral gap estimates for $\mathcal{L}_{0}(\eta)$ from \cite[Lemma A.2.1]{duncan}. That is, there exists a constant $\lambda(\eta)>0$ such that for all $f\in L^{2}(\rho_{Y}(\cdot,\eta))$,
\begin{align*}
\norm{P_{t}^{0}f-\langle f\rangle_{\rho_{Y}(\cdot,\eta)}}_{L^{2}(\rho_{Y}(\cdot,\eta))}\leqslant e^{-\lambda(\eta) t}\norm{f}_{L^{2}(\rho_{Y}(\cdot,\eta))},
\end{align*}
where $(P^{0}_{t})_{t\geqslant 0}=(P^{0}_{t}(\eta))_{t\geqslant 0}$ denotes the semigroup on $L^{2}(\rho_{Y}(\cdot,\eta))$ associated to $\mathcal{L}_{0}(\eta)$. Furthermore, according to \cite[Proposition A.2.2]{duncan}, the solution $\chi$ will be smooth in the $\eta$-variable (as $F$ is smooth) and satisfies $\abs{\nabla_{\eta}^{k}\chi(y,\eta)}\leqslant C_{k} (1+\abs{\eta}^{l_{k}})$ for $k=0,1,2$ and $l_{k}\geqslant 1$ and a constant $C_{k}>0$ (due to $F$ satisfying such a growth condition with a possibly different constant $C_{k}$).\\ 
Then, we can decompose the drift part of $X^{\epsilon}$ with the help of that solution $\chi$, yielding 
\begin{align}\label{eq:driftF}
\epsilon^{-1}\int_{0}^{t}F^{i}(Y^{\epsilon}_{s},\eta^{\epsilon}_{s})ds
&=\epsilon({\chi^{i}(Y^{\epsilon}_{0},\eta^{\epsilon}_{0})-\chi^{i}(Y_{t}^{\epsilon},\eta_{t}^{\epsilon})})+\sqrt{\epsilon}\int_{0}^{t}\nabla_{\eta}\chi^{i}(Y^{\epsilon}_{r},\eta^{\epsilon}_{r})\cdot\sqrt{2\Gamma\Pi}dW_{r}\nonumber\\
&\quad+\int_{0}^{t}\nabla_{y}\chi^{i}(Y^{\epsilon}_{r},\eta^{\epsilon}_{r})\cdot\sqrt{2\Sigma(Y^{\epsilon}_{r},\eta^{\epsilon}_{r})}dB_{r}\\
&\quad+\int_{0}^{t}(\mathcal{L}_{\eta}\chi)^{i} (Y^{\epsilon}_{s},\eta^{\epsilon}_{s})ds\nonumber
\end{align} for $i=1,2$.
Plugging \eqref{eq:driftF} into the dynamics for $X^{\epsilon}$, one can deduce that $(X^{\epsilon})$ converges in distribution to a Brownian motion with variance $2Dt$ plus a constant drift $L$. This was proven in \cite[Theorem 5.2.2]{duncan}, namely
\begin{equation*}
(X^{\epsilon}_{t})_{t\geqslant 0}\Rightarrow (tL +\sqrt{2D}Z_{t})_{t\geqslant 0}
\end{equation*}
for a standard Brownian motion $Z$ and where
\begin{align*}
D= \int\int (I + \nabla_y \chi)^{T}\Sigma(I + \nabla_y \chi)\rho_{Y}(dy, \eta)\rho_{\eta}(d\eta)
\end{align*} and 
\begin{align*}
L=\int\int\mathcal{L}_{\eta}\chi(y,\eta)\rho_{Y}(dy,\eta)\rho_{\eta}(d\eta).
\end{align*}
Indeed, the drift term $L$ arises from the convergence of the last term in \eqref{eq:driftF}, that is
\begin{align}\label{eq:L}
\int_{0}^{t} \mathcal{L}_{\eta}\chi(Y^{\epsilon}_{s},\eta^{\epsilon}_{s})ds\to t\int\int\mathcal{L}_{\eta}\chi(y,\eta)\rho_{Y}(dy,\eta)\rho_{\eta}(d\eta)=:tL
\end{align} in probability, which motivates the ergodic theorem for $(Y^{\epsilon},\eta^{\epsilon})$, \cref{prop:mainprop}, below.

\noindent For completeness, we state the result here, its proof follows from the result from \cite[Lemma A.2.3]{duncan}. One can also deduce its claim from decomposing the additive functional in terms of the solution $G(\cdot,\eta)$ of the Poisson equation (analogously as in \eqref{eq:driftF})
\begin{align*}
\mathcal{L}_{0}G(\cdot,\eta)=b (\cdot,\eta)-E_{\rho_{Y}(\cdot,\eta)}[b(\cdot,\eta)]
\end{align*} for fixed $\eta\in\R^{2K}$ (existence follows by the $L^{2}(\rho_{Y}(\cdot,\eta))$-spectral gap estimates on $\mathcal{L}_{0}(\eta)$, \cite[Lemma A.2.1]{duncan}) and utilizing the ergodic theorem for the Ornstein Uhlenbeck process $(\eta_{t})_{t\geqslant 0}$ started in $\rho_{\eta}$. The growth assumption on $b$ in the following proposition is needed to obtain an analogue growth condition on $G$, such that the martingale term and the drift part involving $\mathcal{L}_{\eta}G$ in the decomposition vanish in $L^{2}(\p)$. 
\begin{proposition}\label{prop:mainprop}
Let $b:\mathbb{T}^{2}\times\R^{2K}\to\R$ be such that 
$b(y,\cdot)\in C^{2}(\R^{2K})$ satisfies the growth assumption
 \begin{equation}\label{bGrowth}
 \abs{\nabla_{\eta}^{k}b(y,\eta)}\lesssim_{k} 1+\abs{\eta}^{l_{k}}
\end{equation}
 for some $l_{k}\geqslant 1$ and all $k=0,1,2$.\\
Then the following convergence in probability holds true when $\epsilon\to 0$
\begin{align}
\int_{0}^{t}b(Y^{\epsilon}_{s},\eta^{\epsilon}_{s})ds\to t\int E_{\rho_{Y}(\cdot,\eta)}[b(\cdot,\eta)]_{\mid \eta=\tilde{\eta}}\rho_{\eta}(d\tilde{\eta}).
\end{align} 
(Here $E_{\rho}[\cdot]$ denotes integration with respect to a probability measure $\rho$ on $\mathbb{T}^{2}$, respectively $\R^{2K}$.)\\
\end{proposition}

\begin{remark}
In particular this applies for $b(\cdot,\tilde{\eta}):=\mathcal{L}_{\eta}\chi(\cdot,\tilde{\eta})$, where $\mathcal{L}_{\eta}$ and $\chi$ are defined by \eqref{OUgen} and  \eqref{eq:eta}. This is due to the fact that the derivatives in $\eta$ of the solution $\chi$ also satisfy a growth condition \eqref{bGrowth}. This was proven in \cite[Proposition A.2.2]{duncan} as $F$ satisfies $\abs{(\nabla_{\eta})^{m}F(y,\eta)}\lesssim 1+\abs{\eta}^{q_{m}}$ for some $q_{m}\geqslant 0$.
\end{remark}

\begin{subsection}{Determining the limit rough path}
Similarly as in Section \ref{4.1}, we will represent $X^\epsilon_t$ via the solution of the Poisson equation \eqref{eq:eta}. More precisely, from   \eqref{eq:driftF} we obtain
\begin{align}\label{eq:decomp}
X^{\epsilon,i}_{t}-X^{i}_{0}&=\epsilon({\chi^{i}(Y_{0}^{\epsilon},\eta_{0}^{\epsilon})-\chi^{i}(Y_{t}^{\epsilon},\eta_{t}^{\epsilon})})+\sqrt{\epsilon}\int_{0}^{t}(\nabla_{\eta}\chi^{i}(Y^{\epsilon}_{r},\eta^{\epsilon}_{r}))^{T}\sqrt{2\Gamma\Pi}dW_{r}\nonumber\\
&\quad+\int_{0}^{t}(e_{i}+\nabla_{y}\chi^{i}(Y^{\epsilon}_{r},\eta^{\epsilon}_{r}))^{T}\sqrt{2\Sigma(Y^{\epsilon}_{r},\eta^{\epsilon}_{r})}dB_{r}\\
&\quad+\int_{0}^{t}(\mathcal{L}_{\eta}\chi)^{i} (Y^{\epsilon}_{s},\eta^{\epsilon}_{s})ds\nonumber
\end{align} for $i=1,2$.
We utilize the decomposition \eqref{eq:decomp} to represent 
 the iterated Itô-integrals as follows for $i,j\in\{1,2\}$,
\begin{align}
\mathbb{X}^{\epsilon}_{0,t}(i,j)
&=\int_{0}^{t}(X^{\epsilon,i}_{s}-X^{i}_{0}) dX^{\epsilon,j}_{s}\nonumber\\&=\int_{0}^{t}({\chi^{i}(Y_{0}^{\epsilon},\eta_{0}^{\epsilon})-\chi^{i}(Y_{s}^{\epsilon},\eta_{s}^{\epsilon})}) F^{j}(Y^{\epsilon}_{s},\eta^{\epsilon}_{s})ds\label{eq:1}\\
&\quad +\sum_{l=1,2}\int_{0}^{t}\epsilon({\chi^{i}(Y_{0}^{\epsilon},\eta_{0}^{\epsilon})-\chi^{i}(Y_{s}^{\epsilon},\eta_{s}^{\epsilon})})\sqrt{2\Sigma(Y^{\epsilon}_{s},\eta^{\epsilon}_{s})}(j,l) dB_{s}^{l}\label{eq:2}\\
&\quad+\sum_{l=1,2}\int_{0}^{t}\sqrt{\epsilon}\paren[\bigg]{\int_{0}^{s}\nabla_{\eta}\chi^{i}(Y^{\epsilon}_{r},\eta^{\epsilon}_{r})\cdot\sqrt{2\Gamma\Pi}dW_{r}}\sqrt{2\Sigma(Y^{\epsilon}_{s},\eta^{\epsilon}_{s})}(j,l) dB_{s}^{l}\label{eq:3}\\
&\quad+\int_{0}^{t}\epsilon^{-1/2}\paren[\bigg]{\int_{0}^{s}{\nabla_{\eta}\chi^{i}}(Y^{\epsilon}_{r},\eta^{\epsilon}_{r})\cdot\sqrt{2\Gamma\Pi}dW_{r}} F^{j}(Y^{\epsilon}_{s},\eta^{\epsilon}_{s})ds\label{eq:4}\\
&\quad \squeeze[0.0001]{+\sum_{l=1,2}\int_{0}^{t}\!\!\paren[\bigg]{\int_{0}^{s}\!\!\!(e_{i}+\nabla_{y}\chi^{i}(Y^{\epsilon}_{r},\eta^{\epsilon}_{r}))\cdot\sqrt{2\Sigma(Y^{\epsilon}_{r},\eta^{\epsilon}_{r})}dB_{r}+\int_{0}^{s}\!\!\!(\mathcal{L}_{\eta}\chi)^{i}(Y^{\epsilon}_{r},\eta^{\epsilon}_{r})dr}\sqrt{2\Sigma(Y^{\epsilon}_{s},\eta^{\epsilon}_{s})}(j,l)dB_{s}^{l}\label{eq:5}}\\
&\quad  \squeeze[0.0001]{+\int_{0}^{t}\epsilon^{-1}\paren[\bigg]{\int_{0}^{s}(e_{i}+\nabla_{y}\chi^{i}(Y^{\epsilon}_{r},\eta^{\epsilon}_{r}))\cdot\sqrt{2\Sigma(Y^{\epsilon}_{r},\eta^{\epsilon}_{r})}dB_{r}+\int_{0}^{s}(\mathcal{L}_{\eta}\chi)^{i}(Y^{\epsilon}_{r},\eta^{\epsilon}_{r})dr} F^{j}(Y^{\epsilon}_{s},\eta^{\epsilon}_{s})ds}.\label{eq:6}
\end{align}
We immediately see that  terms \eqref{eq:2} and \eqref{eq:3} will converge in $L^{2}(\p)$ to zero by Burkholder-Davis-Gundy inequality and the growth conditions on $\nabla_{\eta}^{k}\chi$ for $k=0,1$ by \cite[Proposition A.2.2]{duncan}. Furthermore, the fourth term \eqref{eq:4} can be written by integration by parts as
\begin{align*}
\MoveEqLeft
\int_{0}^{t}\epsilon^{-1/2}\paren[\bigg]{\int_{0}^{s}\nabla_{\eta}\chi^{i}(Y^{\epsilon}_{r},\eta^{\epsilon}_{r})\cdot\sqrt{2\Gamma\Pi}dW_{r}}F^{j}(Y^{\epsilon}_{s},\eta^{\epsilon}_{s})ds
\\&=\int_{0}^{t}\paren[\bigg]{\int_{0}^{s}\epsilon^{1/2}\nabla_{\eta}\chi^{i}(Y^{\epsilon}_{r},\eta^{\epsilon}_{r})\cdot\sqrt{2\Gamma\Pi}dW_{r}}\epsilon^{-1}F^{j}(Y^{\epsilon}_{s},\eta^{\epsilon}_{s})ds\\
&=\int_{0}^{t}\paren[\bigg]{\int_{0}^{s}\epsilon^{-1}F^{j}(Y^{\epsilon}_{s},\eta^{\epsilon}_{s})ds}d\paren[\bigg]{\int_{0}^{\cdot}\epsilon^{1/2}\nabla_{\eta}\chi^{i}(Y^{\epsilon}_{r},\eta^{\epsilon}_{r})\cdot\sqrt{2\Gamma\Pi}dW_{r}}_{s}\\
&\quad+\paren[\bigg]{\int_{0}^{t}\epsilon^{-1}F^{j}(Y^{\epsilon}_{s},\eta^{\epsilon}_{s})ds}\paren[\bigg]{\int_{0}^{t}\epsilon^{1/2}\nabla_{\eta}\chi^{i}(Y^{\epsilon}_{s},\eta^{\epsilon}_{s})\cdot\sqrt{2\Gamma\Pi}dW_{s}}.
\end{align*} 
and utilizing \cref{prop:UCV}, we want to deduce that the term \eqref{eq:4} converges to zero in probability. We have that
\begin{equation*}
\paren[\bigg]{\epsilon^{-1}\int_{0}^{\cdot} F^{j}(Y^{\epsilon}_{s},\eta^{\epsilon}_{s})ds,\int_{0}^{\cdot}\epsilon^{1/2}\nabla_{\eta}\chi^{i}\sqrt{2\Gamma\Pi}((Y^{\epsilon}_{s},\eta^{\epsilon}_{s})dW_{s}}\Rightarrow \paren[\big]{(\tilde{Z}^{j}_{t}+tL^{j})_{t},0}
\end{equation*}
jointly in distribution by the decomposition \eqref{eq:driftF} and the arguments from \cite[Theorem 5.2.2.]{duncan}. Here $\tilde{Z}$ is the limiting Brownian motion with variance $t\int\int(\nabla_{y}\chi)^{T}2\Sigma\nabla_{y}\chi d\rho_{Y}d\rho_{\eta}$. Furthermore, the UCV condition holds for the martingale, as $\nabla_{\eta}\chi\in L^{2}(\rho_{Y}(\eta)\rho_{\eta})$ by the growth condition proven in \cite[Proposition A.2.2]{duncan}, with which the expected quadratic variation can be bounded. Together this then yields that \eqref{eq:4} converges to zero in probability.  

\noindent Applying  \cref{prop:mainprop}  for $b=\chi^{i} F^{j}$ and using that $\int F(y,\eta)d\rho_{Y}(dy,\eta)=0$ by the definition of $F$ and $\rho_{Y}$, we obtain that the first term \eqref{eq:1} will converge in probability to 
\begin{equation*}
t\int\int\chi^{i} F^{j}\rho_{Y}(dy,\eta)\rho_{\eta}(d\eta)=:ta_{F}(i,j).
\end{equation*} 
In order to deal with the remaining   terms \eqref{eq:5} and \eqref{eq:6}, we rewrite them by the integration by parts, respectively Itô's formula. For \eqref{eq:6} we obtain by integration by parts
\begin{align*}
\MoveEqLeft
\int_{0}^{t}\epsilon^{-1}\paren[\bigg]{\int_{0}^{s}(e_{i}+\nabla_{y}\chi^{i}(Y^{\epsilon}_{r},\eta^{\epsilon}_{r}))\cdot\sqrt{2\Sigma(Y^{\epsilon}_{r},\eta^{\epsilon}_{r})}dB_{r}+\int_{0}^{s}(\mathcal{L}_{\eta}\chi)^{i}(Y^{\epsilon}_{r},\eta^{\epsilon}_{r})dr}F^{j}(Y^{\epsilon}_{s},\eta^{\epsilon}_{s})ds\\
&=(\int_{0}^{t}\epsilon^{-1}F^{j}(Y^{\epsilon}_{s},\eta^{\epsilon}_{s}))ds)\paren[\bigg]{\int_{0}^{t}(e_{i}+\nabla_{y}\chi^{i}(Y^{\epsilon}_{s},\eta^{\epsilon}_{s}))\cdot\sqrt{2\Sigma(Y^{\epsilon}_{s},\eta^{\epsilon}_{s})}dB_{s}+\int_{0}^{t}(\mathcal{L}_{\eta}\chi)^{i}(Y^{\epsilon}_{s},\eta^{\epsilon}_{s})ds}\\
&\quad-\int_{0}^{t}(\int_{0}^{s}\epsilon^{-1}F^{j}(Y^{\epsilon}_{r},\eta^{\epsilon}_{r})dr)d\paren[\bigg]{\int_{0}^{\cdot}(e_{i}+\nabla_{y}\chi^{i}(Y^{\epsilon}_{r},\eta^{\epsilon}_{r}))\cdot\sqrt{2\Sigma(Y^{\epsilon}_{r},\eta^{\epsilon}_{r})}dB_{r})+\int_{0}^{\cdot}(\mathcal{L}_{\eta}\chi)^{i}(Y^{\epsilon}_{r},\eta^{\epsilon}_{r})dr}_{s}.
\end{align*}
According to Itô's formula  for \eqref{eq:5} we have
\begin{align*}
\MoveEqLeft
\int_{0}^{t}\paren[\bigg]{\int_{0}^{s}(e_{i}+\nabla_{y}\chi^{i}(Y^{\epsilon}_{r},\eta^{\epsilon}_{r}))\cdot\sqrt{2\Sigma(Y^{\epsilon}_{r},\eta^{\epsilon}_{r})}dB_{r}+\int_{0}^{s}(\mathcal{L}_{\eta}\chi)^{i}(Y^{\epsilon}_{r},\eta^{\epsilon}_{r})dr}d\left(\int_{0}^{\cdot} \sqrt{2\Sigma(Y^{\epsilon}_{r},\eta^{\epsilon}_{r})}dB_{r} \right)^{j}_{s}\\
&=\paren[\bigg]{\int_{0}^{t}(e_{i}+\nabla_{y}\chi^{i}(Y^{\epsilon}_{s},\eta^{\epsilon}_{s}))\cdot\sqrt{2\Sigma(Y^{\epsilon}_{s},\eta^{\epsilon}_{s})}dB_{s}+\int_{0}^{t}(\mathcal{L}_{\eta}\chi)^{i}(Y^{\epsilon}_{s},\eta^{\epsilon}_{s})ds}\left(\int_{0}^{t} \sqrt{2\Sigma(Y^{\epsilon}_{s},\eta^{\epsilon}_{s})}dB_{s}\right)^{j}\\
&\quad-\int_{0}^{t}\left(\int_{0}^{s}\sqrt{2\Sigma(Y^{\epsilon}_{r},\eta^{\epsilon}_{r})}dB_{r}\right)^{j}d\paren[\bigg]{\int_{0}^{\cdot}(e_{i}+\nabla_{y}\chi^{i}(Y^{\epsilon}_{r},\eta^{\epsilon}_{r}))\cdot\sqrt{2\Sigma(Y^{\epsilon}_{r},\eta^{\epsilon}_{r})}dB_{r}+\int_{0}^{\cdot}(\mathcal{L}_{\eta}\chi)^{i}(Y^{\epsilon}_{r},\eta^{\epsilon}_{r})dr}_{s}\\
&\quad-\int_{0}^{t}(e_{i}+\nabla_{y}\chi^{i}(Y^{\epsilon}_{s},\eta^{\epsilon}_{s}))\cdot2\Sigma(Y^{\epsilon}_{s},\eta^{\epsilon}_{s}) e_{j}ds.
\end{align*} 
By adding the previous two terms up, we obtain that overall
\begin{align}
\MoveEqLeft
\int_{0}^{t}X^{\epsilon,i}_{s}dX^{\epsilon,j}_{s}\nonumber\\
&= a_{F}^{\epsilon}(i,j)+X^{\epsilon,j}_{t}\paren[\bigg]{\int_{0}^{t}(e_{i}+\nabla_{y}\chi^{i}(Y^{\epsilon}_{r},\eta^{\epsilon}_{r}))\cdot\sqrt{2\Sigma(Y^{\epsilon}_{r},\eta^{\epsilon}_{r})}dB_{r}+\int_{0}^{t}(\mathcal{L}_{\eta}\chi)^{i}(Y^{\epsilon}_{r},\eta^{\epsilon}_{r})dr}\nonumber\\
&\quad-\int_{0}^{t}X^{\epsilon,j}_{s}d\paren[\bigg]{\int_{0}^{\cdot}(e_{i}+\nabla_{y}\chi^{i}(Y^{\epsilon}_{r},\eta^{\epsilon}_{r}))\cdot\sqrt{2\Sigma(Y^{\epsilon}_{r},\eta^{\epsilon}_{r})}dB_{r}+\int_{0}^{\cdot}(\mathcal{L}_{\eta}\chi)^{i}(Y^{\epsilon}_{r},\eta^{\epsilon}_{r})dr}_{s}\nonumber\\
&\quad-\int_{0}^{t} (e_{i}+\nabla_{y}\chi^{i}(Y^{\epsilon}_{r},\eta^{\epsilon}_{r}))\cdot 2\Sigma(Y^{\epsilon}_{r},\eta^{\epsilon}_{r}) e_{j}dr\label{eq:boldX},
\end{align}
where $a_{F}^{\epsilon}(i,j)$ denotes the sum of the terms \eqref{eq:1},\eqref{eq:2},\eqref{eq:3} and \eqref{eq:4}, that will converge in probability to 
 \begin{equation*}
 ta_{F}(i,j):=t\int\int\chi^{i} F^{j}\rho_{Y}(dy,\eta)\rho_{\eta}(d\eta).
 \end{equation*}
For the stochastic integral and the product term in \eqref{eq:boldX}, we again apply \cref{prop:UCV} with  
\begin{align*}
\left(X^{\epsilon,j},\int_{0}^{\cdot}(e_{i}+\nabla_{y}\chi^{i}(Y^{\epsilon}_{r},\eta^{\epsilon}_{r}))\cdot\sqrt{2\Sigma(Y^{\epsilon}_{r},\eta^{\epsilon}_{r})}dB_{r}+\int_{0}^{\cdot}(\mathcal{L}_{\eta}\chi)^{i}(Y^{\epsilon}_{r},\eta^{\epsilon}_{r})dr \right)\Rightarrow (X^{j},X^{i})
\end{align*} 
jointly in distribution (by \cite[Theorem 5.2.2]{duncan}) using that the  UCV condition for the semi-martingales
 $$\paren[\bigg]{\int_{0}^{\cdot}(e_{i}+\nabla_{y}\chi^{i}(Y^{\epsilon}_{r},\eta^{\epsilon}_{r}))\cdot\sqrt{2\Sigma(Y^{\epsilon}_{r},\eta^{\epsilon}_{r})}dB_{r}+\int_{0}^{\cdot}(\mathcal{L}_{\eta}\chi)^{i}(Y^{\epsilon}_{r},\eta^{\epsilon}_{r})dr}_{\epsilon}$$
 holds (by boundedness of $\Sigma$ and the bounds (A.7) and (A.8) of \cite[Proposition A.2.2]{duncan}, which are bounds for the expected quadratic variation). Let us define
\begin{equation*}
c^{i,j}:=\int\int (e_{i}+\nabla_{y}\chi^{i})\cdot 2\Sigma e_{j}d\rho_{Y}(\cdot,\eta)d\rho_{\eta}.
\end{equation*}
Then according to \cref{prop:mainprop}, the last term in \eqref{eq:boldX} (the quadratic variation term) converges in probability to $tc^{i,j}$.
Hence, together, utilizing Slutzky's Lemma, we obtain the distributional convergence
\begin{align*}
\MoveEqLeft
\int_{0}^{t}X^{\epsilon,i}_{s}dX^{\epsilon,j}_{s}\\
&\Rightarrow t a_{F}(i,j)+X^{j}_{t}X^{i}_{t}-\int_{0}^{t}X^{j}_{s}dX^{i}_{s}-\frac{t}{2}c^{i,j}\\
&=t(a_{F}(i,j)-c^{i,j}+\langle X^{i},X^{j}\rangle_{1})+\int_{0}^{t}X^{i}_{s}dX^{j}_{s}.
\end{align*}

\noindent Let us summarize our findings about the convergence of the iterated Itô integrals, as well as the iterated Stratonovich integrals, in the next proposition and corollary.
\begin{proposition}\label{prop:limit3}
Let $(X^{\epsilon},Y^{\epsilon},\eta^{\epsilon})$ be the solution of the system \eqref{eq:system3} for $(Y^{\epsilon}_{0},\eta^{\epsilon}_{0})\sim\rho_{Y}(dy,\eta)\rho_{\eta}(d\eta)$. Moreover, let 
\begin{align*}
X_{t}=\sqrt{2D}Z_{t}+tL
\end{align*}
for a standard two-dimensional Brownian motion $Z$,  and
\begin{align*}
D=\int\int (I + \nabla_y \chi(y,\eta))^{T}{\Sigma(y,\eta)}(I + \nabla_y \chi(y,\eta)) \rho_{Y}(dy, \eta)\rho_{\eta}(d\eta)
\end{align*} and 
\begin{align*}
L=\int\int\mathcal{L}_{\eta}\chi(y,\eta)\rho_{Y}(dy, \eta)\rho_{\eta}(d\eta),
\end{align*}
 where for each $\eta\in\R^{K}$, $\chi(\cdot,\eta)$ is the solution of  $\mathcal{L}_{0}\chi(\cdot,\eta)=-F(\cdot,\eta)$.

\noindent Then for all $s,t\in\Delta_{T}$, weak convergence in $\R^{2\times 2}$ of the iterated Itô-integrals $(\mathbb{X}^{\epsilon}_{s,t})$ holds true, where for $i,j\in\{1,2\}$
\begin{align}
\mathbb{X}^{\epsilon}_{s,t}(i,j):=\int_{s}^{t}X^{\epsilon,i}_{s,r}dX^{\epsilon,j}_{r}
&\Rightarrow \mathbb{X}_{s,t}(i,j)+t\paren[\bigg]{\langle X^{i},X^{j}\rangle_{1}+\langle\chi^{i},\mathcal{L}_{0}\chi^{j}\rangle_{\rho_{Y}(\cdot,\eta)\rho_{\eta}}\nonumber\\&\quad-\int\int (e_{i}+\nabla\chi^{i}(y,\eta)\cdot 2\Sigma(y,\eta) e_{j}\rho_{Y}(dy,\eta)\rho_{\eta}(d\eta)}
\end{align} 
as $\epsilon\to 0$, where $\mathbb{X}_{s,t}(i,j):=\int_{s}^{t}X_{s,r}^{i}dX^{j}_{r}$.
\end{proposition}
\begin{corollary}\label{cor:limit3}
Let $(X^{\epsilon},Y^{\epsilon},\eta^{\epsilon})$ be {as in \cref{prop:limit3}}. Then for all $s,t\in\Delta_{T}$, weak convergence in $\R^{2\times 2}$ of the iterated Stratonovich-integrals $(\tilde{\mathbb{X}}^{\epsilon}_{s,t})$ holds true, where for $i,j\in\{1,2\}$
\begin{align}
\tilde{\mathbb{X}}^{\epsilon}_{s,t}(i,j):=\int_{s}^{t}X^{\epsilon,i}_{s,r}\circ dX^{\epsilon,j}_{r}\Rightarrow\tilde{\mathbb{X}}_{s,t}(i,j) + t\tilde{A}(i,j)
\end{align}
weakly for $\epsilon\to 0$, where 
\begin{align*}
X_{t}=\sqrt{2D}Z_{t}+tL
\end{align*} 
for a standard two-dimensional Brownian motion $Z$ and $\tilde{\mathbb{X}}_{s,t}(i,j):=\int_{s}^{t}X_{s,r}^{i}\circ dX^{j}_{r}$ and $D,L$ are defined as in \cref{prop:limit3}. The area correction is given by
\begin{align*}
\tilde{A}(i,j)&=\int\int e_{i}\cdot \Sigma\nabla_{y}\chi^{j}(y,\eta)\rho_{y}(dy,\eta)\rho_{\eta}(d\eta)-\int\int\nabla_{y}\chi^{i}(y,\eta)\cdot\Sigma e_{j}\rho_{y}(dy,\eta)\rho_{\eta}(d\eta)\\&=\int\int \chi^{i}(y,\eta)F^{j}(y,\eta)\rho_{y}(dy,\eta)\rho_{\eta}(d\eta) - \int\int F^{i}(y,\eta)\chi^{j}(y,\eta)\rho_{y}(dy,\eta)\rho_{\eta}(d\eta)
\end{align*}
\end{corollary}
\begin{proof}
The corollary follows from \cref{prop:limit3} and analogue arguments as in \cref{cor:limit1}. We have that
\begin{align*}
{\frac{1}{2}\langle X^{\epsilon,i},X^{\epsilon,j}\rangle_{t}=\int_{0}^{t}e_{i}\cdot\Sigma(Y^{\epsilon}_{s},\eta^{\epsilon}_{s}) e_{j}ds\to t\int e_{i}\cdot\Sigma(y,\eta) e_{j}\rho_{Y}(dy,\eta)\rho_{\eta}(d\eta)}.
\end{align*} 
using \cref{prop:mainprop}.
Furthermore we have that
\begin{align*}
\tilde{\mathbb{X}}^{\epsilon}_{0,t}(i,j)=\mathbb{X}^{\epsilon}_{0,t}(i,j)+\frac{1}{2}\langle X^{\epsilon,i},X^{\epsilon,j}\rangle_{t}&\Rightarrow\mathbb{X}_{0,t}(i,j)+t\langle\chi^{i},(\mathcal{L}_{0}\chi)^{j}\rangle_{\rho_{y}(\cdot,\eta)\rho_{\eta}}+\langle X^{i},X^{j}\rangle_{t}\\&\quad-t\int (\nabla_{y}\chi^{i}+e_{i})\cdot 2\Sigma e_{j}d(\rho_{y}(\cdot,\eta)\rho_{\eta})+t\int e_{i}\cdot \Sigma e_{j}d(\rho_{y}(\cdot,\eta)\rho_{\eta})\\& = \mathbb{X}_{0,t}(i,j)+ \frac{1}{2}\langle X^{i},X^{j}\rangle_{t} \\&\quad +t\paren[\bigg]{\langle\chi^{i},(\mathcal{L}_{0}\chi)^{j}\rangle_{\rho_{y}(\cdot,\eta)\rho_{\eta}} + \frac{1}{2}\langle X^{i},X^{j}\rangle_{t} \\&\quad-\int \nabla_{y}\chi^{i}\cdot 2\Sigma e_{j}d(\rho_{y}(\cdot,\eta)\rho_{\eta}) - \int e_{i}\cdot \Sigma e_{j}d(\rho_{y}(\cdot,\eta)\rho_{\eta})}\\&=\tilde{\mathbb{X}}_{0,t}(i,j)+t\tilde{A}(i,j).
\end{align*} 
The area correction can be written as
\begin{align*}
\tilde{A}(i,j) &= \langle\chi^{i},\mathcal{L}_{0}\chi^{j}\rangle_{\rho_{y}(\cdot,\eta)\rho_{\eta}} + D(i,j) -\int \nabla_{y}\chi^{i}\cdot 2\Sigma e_{j}d(\rho_{y}(\cdot,\eta)\rho_{\eta}) - \int e_{i}\cdot \Sigma e_{j}d(\rho_{y}(\cdot,\eta)\rho_{\eta}) \\&=\int e_{i}\cdot \Sigma \nabla_{y}\chi^{j}d(\rho_{y}(\cdot,\eta)\rho_{\eta})-\int\nabla_{y}\chi^{i}\cdot\Sigma e_{j}d(\rho_{y}(\cdot,\eta)\rho_{\eta}),
\end{align*}
using furthermore that by the definition of $F$ and the invariant measure $\rho_{y}(dy,\eta)=Z^{-1}\sqrt{\abs{\Sigma(y,\eta)}}dy$,
\begin{align*}
\langle\chi^{i},\mathcal{L}_{0}\chi^{j}\rangle_{\rho_{y}(\cdot,\eta)\rho_{\eta}}
=-\int \nabla_{y}\chi^{i}\cdot\Sigma\nabla_{y}\chi^{j}d(\rho_{y}(\cdot,\eta)\rho_{\eta}). 
\end{align*}
By integrating $\nabla_{y}$ by parts and using once more the definition of $F$ and the invariant measure, we obtain that
\begin{align*}
\int \nabla_{y}\chi^{i}\cdot \Sigma e_{j}d(\rho_{y}(\cdot,\eta)\rho_{\eta}) = - \int_{\R^{2K}} \int_{\mathbb{T}^{2}} \chi^{i}(y,\eta)F^{j}(y,\eta)\rho_{y}(dy,\eta)\rho_{\eta}(d\eta),
\end{align*}
such that the claim for $\tilde{A}$ follows.
\end{proof}
\begin{corollary}\label{cor:area-vanish}
For $\tilde{A}$ from \cref{cor:limit3}, it follows that $\tilde{A}(i,j)=0$ for all $i,j=1,2$.
\end{corollary}
\begin{proof}
Using that $(-\mathcal{L}_{0})\chi=F$, we obtain
\begin{align*}
\tilde{A}(i,j)&=\int\int \chi^{i}(y,\eta)F^{j}(y,\eta)\rho_{y}(dy,\eta)\rho_{\eta}(d\eta) - \int\int F^{i}(y,\eta)\chi^{j}(y,\eta)\rho_{y}(dy,\eta)\rho_{\eta}(d\eta)\\&=\int\int \chi^{i}(y,\eta)(-\mathcal{L}_{0})\chi^{j}(y,\eta)\rho_{y}(dy,\eta)\rho_{\eta}(d\eta) - \int\int (-\mathcal{L}_{0})\chi^{i}(y,\eta)\chi^{j}(y,\eta)\rho_{y}(dy,\eta)\rho_{\eta}(d\eta)\\&=\int\int (-\mathcal{L}_{0}^{\star})\chi^{i}(y,\eta)\chi^{j}(y,\eta)\rho_{y}(dy,\eta)\rho_{\eta}(d\eta) - \int\int (-\mathcal{L}_{0})\chi^{i}(y,\eta)\chi^{j}(y,\eta)\rho_{y}(dy,\eta)\rho_{\eta}(d\eta)\\&=0
\end{align*}
and the claim follows as $\mathcal{L}_{0}$ is symmetric with respect to the measure $\rho_{y}(dy,\eta)\rho_{\eta}(d\eta)$, that is $\mathcal{L}_{0}^{\star}=\mathcal{L}_{0}$, where $\mathcal{L}_{0}^{\star}$ denotes the adjoint with respect to $L^{2}(\rho_{y}(dy,\eta)\rho_{\eta}(d\eta))$.
\end{proof}
\begin{remark}\label{rem:conjecture}
Note that for the limit of the Stratonovich integrals no area correction appears. This is due to the fact that the underlying Markov process $(Y_{t})_{t\geqslant 0}$ (for fixed $\eta\in\R^{2K}$) is reversible when started in $\rho_{Y}(\cdot,\eta)$ (meaning that the generator $\mathcal{L}_{0}(\eta)=\mathcal{L}_{0}(\eta)^{\star}$ is symmetric with respect to $L^{2}(\rho_{Y}(\cdot,\eta))$). This is a phenomenon that is observed generally, see also the discussion in the introduction of \cite{radditive} about the relation of a vanishing Stratonovich area correction and an underlying reversible Markov process. The conjecture from \cite{radditive} states that the Stratonovich area correction vanishes if and only if the underlying Markov process is reversible. 
\end{remark}
\begin{remark}
Via a symmetry argument utilizing the Fourier expansion of the Helfrich surface $H$, one can show that the limiting drift $L$ actually vanishes, $L=0$. For  proof see \cite[Proposition 5.2.4]{duncan}.
\end{remark}
\end{subsection}

\begin{subsection}{Tightness in $C_{\gamma,T}$ for $\gamma<1/2$}
The tightness is again a consequence of \cref{lem:rp-kol}, once we have verified the moment bounds.
\begin{proposition}\label{prop:tight3}
Let $(X^{\epsilon},Y^{\epsilon},\eta^{\epsilon})$ be as in \cref{prop:limit3}. Then the following moment bounds hold true for any $p\geqslant 2$
\begin{align}\label{eq:xest}
\sup_{\epsilon}\E[\abs{X^{\epsilon,i}(t)-X^{\epsilon,i}(s)}^{p}]\lesssim\abs{t-s}^{{p/2}},\quad\forall s,t\in\Delta_{T}
\end{align} and
\begin{align}
\sup_{\epsilon}\E\bigg[\abs[\bigg]{\int_{s}^{t}X^{\epsilon, i}_{s,r} dX^{\epsilon,j}_{r}}^{p/2}\bigg]\lesssim\abs{t-s}^{{p/2}},\quad\forall s,t\in\Delta_{T}
\end{align} 
for $i,j\in\{1,2\}$.\\
In particular, tightness of $(X^{\epsilon},\mathbb{X}^{\epsilon})$ in $C_{\gamma,T}$ for $\gamma<1/2$ follows.
\end{proposition}
\begin{proof}
The arguments are analogous as for \cref{prop:tight1}. For the estimate for $X^{\epsilon}$ we use, similarly as in \cref{prop:tight1}, a trade-off argument for the drift term using the decomposition \eqref{eq:decomp}. For the iterated integrals, we use the bound \eqref{eq:xest} on $X^{\epsilon}$ and the decomposition \eqref{eq:boldX}. 
\end{proof}

\end{subsection}

\begin{subsection}{Rough homogenization limit in the $(\alpha,\beta)=(1,1)$-regime}
In this subsection, we conclude on our second main theorem, which is a corollary from \cref{prop:limit3} and \cref{prop:tight3}. The corollary on the Statonovich lift then follows from the result for the Itô lift (\cref{thm:main-thm2} below), \cref{cor:limit3} and \cref{cor:area-vanish}.
\begin{theorem}[Itô-lift]\label{thm:main-thm2}
Let $(X^{\epsilon},Y^{\epsilon},\eta^{\epsilon})$, $X$ and $\mathbb{X}^{\epsilon},\mathbb{X}$ be as in \cref{prop:limit3}. Then for all $\gamma<1/2$, the weak convergence in the rough path space $C_{\gamma,T}$ of
\begin{align}
(X^{\epsilon},\mathbb{X}^{\epsilon})\Rightarrow (X,(s,t)\mapsto\mathbb{X}_{s,t}+A(t-s))
\end{align} 
follows when $\epsilon\to 0$ and where $A=(A(i,j))_{i,j\in\{1,2\}}$ denotes the matrix with elements
\begin{align}
A(i,j)=\langle\chi^{i},\mathcal{L}_{0}\chi^{j}\rangle_{\rho_{Y}(\cdot,\eta)\rho_{\eta}}+\langle X^{i},X^{j}\rangle_{1}-\int\int (e_{i}+\nabla_{y}\chi^{i}(y,\eta))\cdot 2\Sigma(y,\eta) e_{j}\rho_{Y}(dy,\eta)\rho_{\eta}(d\eta).
\end{align} 
\end{theorem}
\begin{corollary}[Stratonovich-lift]
Let $(X^{\epsilon},Y^{\epsilon},\eta^{\epsilon})$, $X$ and $\tilde{\mathbb{X}}^{\epsilon},\tilde{\mathbb{X}}$ be as in \cref{cor:limit3}. Then for all $\gamma<1/2$, the weak convergence in the rough paths space $C_{\gamma,T}$ of
\begin{align}
(X^{\epsilon},\tilde{\mathbb{X}}^{\epsilon}) \Rightarrow (X,\tilde{\mathbb{X}})
\end{align} follows when $\epsilon\to 0$.
\end{corollary}

\end{subsection}

\end{section}

\begin{section}*{Acknowledgements}
A. Dj.     was partially  funded by Deutsche Forschungsgemeinschaft (DFG)
    through the grant CRC 1114: “Scaling Cascades in Complex Systems”,
    Project Number 235221301 and partially by DFG under
Germany’s Excellence Strategy – The Berlin Mathematics Research Center MATH+ (EXC-2046/1,
project ID: 390685689).\\
H.K. was funded by the DFG under
Germany’s Excellence Strategy – The Berlin Mathematics Research Center MATH+ (EXC-2046/1,
project ID: 390685689).\\
We gratefully thank Peter Friz for fruitful discussions on the topic.

\end{section}

\bibliographystyle{alpha}
\bibliography{rough-homogenization-references}

\newcommand{\etalchar}[1]{$^{#1}$}
\begin{thebibliography}{CFK{\etalchar{+}}19}

\bibitem[AV95]{almeida1995lateral}
Paulo~FF Almeida and Winchil~LC Vaz.
\newblock Lateral diffusion in membranes.
\newblock In {\em Handbook of biological physics}, volume~1, pages 305--357.
  Elsevier, 1995.

\bibitem[Can70]{canham1970minimum}
Peter~B Canham.
\newblock The minimum energy of bending as a possible explanation of the
  biconcave shape of the human red blood cell.
\newblock {\em Journal of theoretical biology}, 26(1):61--81, 1970.

\bibitem[CFK{\etalchar{+}}19]{rhfriz}
Ilya Chevyrev, Peter~K. Friz, Alexey Korepanov, Ian Melbourne, and Huilin
  Zhang.
\newblock Multiscale systems, homogenization, and rough paths.
\newblock In {\em Probability and analysis in interacting physical systems},
  volume 283 of {\em Springer Proc. Math. Stat.}, pages 17--48. Springer, Cham,
  2019.

\bibitem[CM78]{Dellacherie-Meyer}
Dellacherie Claude and Paul-André Meyer.
\newblock {\em Probabilities and Potential, A}, volume~29 of {\em North-Holland
  Mathematics Studies}.
\newblock North Holland, 1978.

\bibitem[DE88]{doi1988theory}
Masao Doi and Samuel~Frederick Edwards.
\newblock {\em The theory of polymer dynamics}, volume~73.
\newblock oxford university press, 1988.

\bibitem[DEPS15]{dep}
A.~B. Duncan, C.~M. Elliott, G.~A. Pavliotis, and A.~M. Stuart.
\newblock A multiscale analysis of diffusions on rapidly varying surfaces.
\newblock {\em J. Nonlinear Sci.}, 25(2):389--449, 2015.

\bibitem[Des15]{deserno2015fluid}
Markus Deserno.
\newblock Fluid lipid membranes: From differential geometry to curvature
  stresses.
\newblock {\em Chemistry and physics of lipids}, 185:11--45, 2015.

\bibitem[DOP21]{radditive}
Jean-Dominique Deuschel, Tal Orenshtein, and Nicolas Perkowski.
\newblock Additive functionals as rough paths.
\newblock {\em The Annals of Probability}, 49 (3):1450--1479, 2021.

\bibitem[DPZ96]{DaPrato1996}
G.~Da~Prato and J.~Zabczyk.
\newblock {\em Ergodicity for infinite-dimensional systems}, volume 229 of {\em
  London Mathematical Society Lecture Note Series}.
\newblock Cambridge University Press, Cambridge, 1996.

\bibitem[Dun13]{duncan}
Andrew Duncan.
\newblock {\em Diffusion on Rapidly-varying Surfaces}.
\newblock PhD thesis, University of Warwick, 2013.

\bibitem[FGL15]{fgl}
Peter Friz, Paul Gassiat, and Terry Lyons.
\newblock Physical {B}rownian motion in a magnetic field as a rough path.
\newblock {\em Trans. Amer. Math. Soc.}, 367(11):7939--7955, 2015.

\bibitem[FH20]{fh}
Peter~K. Friz and Martin Hairer.
\newblock {\em A course on rough paths}.
\newblock Universitext. Springer, Cham, second edition, [2020] \copyright 2020.
\newblock With an introduction to regularity structures.

\bibitem[GT01]{gt}
David Gilbarg and Neil~S. Trudinger.
\newblock {\em Elliptic partial differential equations of second order}.
\newblock Classics in mathematics. Springer, 2001.

\bibitem[Hel73]{helfrich1973elastic}
Wolfgang Helfrich.
\newblock Elastic properties of lipid bilayers: theory and possible
  experiments.
\newblock {\em Zeitschrift f{\"u}r Naturforschung C}, 28(11-12):693--703, 1973.

\bibitem[KLO12]{klo}
Tomasz Komorowski, Claudio Landim, and Stefano Olla.
\newblock {\em Fluctuations in {M}arkov processes}, volume 345 of {\em
  Grundlehren der Mathematischen Wissenschaften [Fundamental Principles of
  Mathematical Sciences]}.
\newblock Springer, Heidelberg, 2012.
\newblock Time symmetry and martingale approximation.

\bibitem[KM16]{kelly2016smooth}
David Kelly and Ian Melbourne.
\newblock Smooth approximation of stochastic differential equations.
\newblock {\em The Annals of Probability}, 44(1):479--520, 2016.

\bibitem[KP91]{kp}
Thomas~G. Kurtz and Philip Protter.
\newblock Weak limit theorems for stochastic integrals and stochastic
  differential equations.
\newblock {\em Ann. Probab.}, 19(3):1035--1070, 1991.

\bibitem[KP96]{kurtz1996weak}
Thomas~G Kurtz and Philip~E Protter.
\newblock Weak convergence of stochastic integrals and differential equations.
\newblock In {\em Probabilistic models for nonlinear partial differential
  equations}, pages 1--41. Springer, 1996.

\bibitem[LL05]{ll}
Antoine Lejay and Terry Lyons.
\newblock On the importance of the {L}\'{e}vy area for studying the limits of
  functions of converging stochastic processes. {A}pplication to
  homogenization.
\newblock In {\em Current trends in potential theory}, volume~4 of {\em Theta
  Ser. Adv. Math.}, pages 63--84. Theta, Bucharest, 2005.

\bibitem[LP13]{luby2013physical}
Kate Luby-Phelps.
\newblock The physical chemistry of cytoplasm and its influence on cell
  function: an update.
\newblock {\em Molecular biology of the cell}, 24(17):2593--2596, 2013.

\bibitem[MHS14]{mourao2014connecting}
M{\'a}rcio~A Mour{\~a}o, Joe~B Hakim, and Santiago Schnell.
\newblock Connecting the dots: the effects of macromolecular crowding on cell
  physiology.
\newblock {\em Biophysical journal}, 107(12):2761--2766, 2014.

\bibitem[MP11]{mika2011macromolecule}
Jacek~T Mika and Bert Poolman.
\newblock Macromolecule diffusion and confinement in prokaryotic cells.
\newblock {\em Current opinion in biotechnology}, 22(1):117--126, 2011.

\bibitem[NB07]{nb}
Ali Naji and Frank L.~H. Brown.
\newblock Diffusion on ruffled membrane surfaces.
\newblock {\em J.Chem.Phys.}, 126, 2007.

\bibitem[PV01]{pv}
E.~Pardoux and A.~Yu. Veretennikov.
\newblock On the {P}oisson equation and diffusion approximation. i.
\newblock {\em The Annals of Probability}, 29(3):1061--1085, 2001.

\bibitem[Sei97]{seifert1997configurations}
Udo Seifert.
\newblock Configurations of fluid membranes and vesicles.
\newblock {\em Advances in physics}, 46(1):13--137, 1997.

\bibitem[Stu10]{Stuart10}
A.~M. Stuart.
\newblock Inverse problems: A bayesian perspective.
\newblock {\em Acta Numerica}, 19:451–559, 2010.

\bibitem[Yos95]{Yosida}
Kösaku Yosida.
\newblock {\em Functional Analysis}, volume~6 of {\em Classics in Mathematics}.
\newblock Springer, Berlin, Heidelberg, 1995.

\end{thebibliography}

\end{document}